\documentclass[10pt]{amsart}
\usepackage{amsmath,amssymb,amscd,mathrsfs,enumerate}

\numberwithin{equation}{section}

\newtheorem{theorem}[equation]{Theorem}
\newtheorem{proposition}[equation]{Proposition}
\newtheorem{lemma}[equation]{Lemma}
\newtheorem{corollary}[equation]{Corollary}

\theoremstyle{definition} 
\newtheorem{definition}[equation]{Definition}
\newtheorem{remark}[equation]{Remark}
\newtheorem{example}[equation]{Example}

\DeclareMathOperator{\Char}{Char}
\DeclareMathOperator{\Diff}{Diff}
\DeclareMathOperator{\Dom}{Dom}
\DeclareMathOperator{\Hom}{Hom}
\DeclareMathOperator{\rk}{rank}
\DeclareMathOperator{\sym}{ \sigma\!\!\!\sigma}
\DeclareMathOperator{\bsym}{\!{}^{\mathit b}\!\!\sym}
\DeclareMathOperator{\Span}{span}
\DeclareMathOperator{\spec}{spec}
\DeclareMathOperator{\Levi}{Levi}

\def\bd{{}^b\!d}
\def\bnabla{{}^b\!\nabla}
\def\bdeebar{{}^b\!\deebar}
\def\dee{\partial}
\def\bdee{{}^b\!\dee}
\def\bT{{}^bT}
\def\bWedge{{}^b{\!}\Wedge}
\def\bpi{{}^b\pi}

\def\A{\mathcal A}
\def\D{\mathcal D}
\def\Dbar{\overline\D}
\def\G{\mathcal G}
\def\Hor{\mathcal H}
\def\K{\mathcal K}
\def\Kbar{\overline \K}
\def\Lie{\mathcal {L}}
\def\M{\mathcal M}
\def\N{\mathcal N}
\def\T{\mathcal T}
\def\Vee{\mathcal V}
\def\Veebar{\overline\Vee}
\def\W{\mathcal W}
\def\Wbar{\overline\W}

\def\Ha{\mathscr H}

\def\C{\mathbb C}

\def\Dee{\mathbb D}
\def\Deebar{\overline\Dee}
\def\R{\mathbb R}

\def\a{\mathfrak a}
\def\Hol{\mathfrak {Hol}}
\def\m{\mathfrak m}
\def\Mero{\mathfrak {Mero}}
\def\rr{\mathfrak r}

\def\Sing{\mathfrak S}

\def\dR{\mathrm {dR}}

\def\im{i}

\def\sign #1#2{\epsilon^{#1}_{#2}}

\def\Wedge{\raise2ex\hbox{$\mathchar"0356$}}
\def\av{\mathrm{av}}

\def\inner{\mathbf i}
\def\Id{I}
\def\ev{\mathrm{ev}}

\def\minus{\backslash}
\def\set#1{\{#1\}}
\def\open#1{\smash[t]{\overset{{}_{\,\,\circ}}{#1}{}}}
\def\ie{i.e.}
\def\e{\varepsilon}
\def\bT{{}^b{\!}T}

\def\deebar{\overline \partial}
\def\deebarb{\overline \partial_b}
\def\Laplacian{\square}

\def\neutral#1{} 
\def\display#1#2{\mbox{\parbox{#1} {#2}}}

\begin{document}

\title{Complex $b$-manifolds}
\author{Gerardo A. Mendoza}
\email{gmendoza@math.temple.edu}
\address{Department of Mathematics\\
Temple University\\
Philadelphia, PA 19122}
\begin{abstract}
A complex $b$-structure on a manifold $\M$ with boundary is an involutive subbundle $\bT^{0,1}\M$ of the complexification of $\bT\M$ with the property that $\C\bT\M = \bT^{0,1}\M + \overline{\bT^{0,1}\M}$ as a direct sum; the interior of $\M$ is a complex manifold. The complex $b$-structure determines an elliptic complex of $b$-operators and induces a rich structure on the boundary of $\M$. We study the cohomology of the indicial complex of the $b$-Dolbeault complex.
\end{abstract}

\maketitle

\section{Introduction}\label{sIntroduction}

A complex $b$-manifold is a smooth manifold with boundary together with a complex $b$-structure. The latter is a smooth involutive subbundle $\bT^{0,1}\M$ of the complexification $\C\bT\M$ of Melrose's $b$-tangent bundle \cite{RBM1,RBM2} with the property that
\begin{equation*}
\C\bT\M = \bT^{0,1}\M + \overline{\bT^{0,1}\M}
\end{equation*}
as a direct sum. Manifolds with complex $b$-structures generalize the situation that arises as a result of spherical and certain anisotropic (not complex) blowups of complex manifolds at a discrete set of points or along a complex submanifold, cf. \cite[Section 2]{Me3}, \cite{Me5}, as well as (real) blow-ups of complex analytic  varieties with only point singularities.

The interior of $\M$ is a complex manifold. Its $\deebar$-complex determines a $b$-elliptic complex, the $\bdeebar$-complex, on sections of the exterior powers of the dual of $\bT^{0,1}\M$, see Section~\ref{sPreliminaries}. The indicial families $\Dbar(\sigma)$ of the $\bdeebar$-operators at a connected component $\N$ of $\partial \M$ give, for each $\sigma$, an elliptic complex, see Section~\ref{sIndicialComplex}. Their cohomology at the various values of $\sigma$ determine the asymptotics at $\N$ of tempered representatives of cohomology classes of the $\bdeebar$-complex, in particular of tempered holomorphic functions. 

Each boundary component $\N$ of $\M$ inherits from $\bT^{0,1}\M$ the following objects in the $C^\infty$ category:
\begin{enumerate}
\item an involutive vector subbundle $\Veebar\subset \C T\N$ such that $\Vee+\Veebar=\C T\N$;
\item a real nowhere vanishing vector field $\T$ such that $\Vee\cap\Veebar=\Span_\C\T$;
\item a class $\pmb \beta$ of sections of $\smash[t]{\Veebar}^*$,
\end{enumerate}
where the elements of $\pmb \beta$ have additional properties, described in (4) below. The vector bundle $\Veebar$, being involutive, determines a complex of first order differential operators $\Deebar$ on sections of the exterior powers of $\smash[t]{\Veebar}^*$, elliptic because of the second property in (1) above. To that list add
\begin{enumerate}
\item [(4)] If $\beta\in \pmb \beta$ then $\Deebar\beta=0$ and $\Im\langle \beta,\T\rangle=-1$, and if $\beta$, $\beta'\in\pmb\beta$, then $\beta'-\beta=\Deebar u$ with $u$ real-valued.
\end{enumerate}
These properties, together with the existence of a Hermitian metric on $\Veebar$ invariant under $\T$ make $\N$ behave in many ways as the circle bundle of a holomorphic line bundle over a compact complex manifold. These analogies are investigated in  \cite{Me6,Me7,Me8,Me9}. The last of these papers contains a detailed account of circle bundles from the perspective of these boundary structures. The paper \cite{Me4}, a predecessor of the present one, contains some facts studied here in more detail.

\medskip
The paper is organized as follows. Section~\ref{sPreliminaries} deals with the definition of complex $b$-structure and Section~\ref{sHolomorphicVectorBundles} with holomorphic vector bundles over complex $b$-manifolds (the latter term just means that the $b$-tangent bundle takes on a primary role over that of the usual tangent bundle). The associated Dolbeault complexes are defined in these sections accordingly.

Section~\ref{sBoundaryStructure} is a careful account of the structure inherited by the boundary.

In Section~\ref{sLocalInvariants} we show that complex $b$-structures have no formal local invariants at boundary points. The issue here is that we do not have a Newlander-Nirenberg theorem that is valid in a neighborhoods of a point of the boundary, so no explicit local model for $b$-manifolds.

Section~\ref{sIndicialComplex} is devoted to general aspects of $b$-elliptic first order complexes $A$. We introduce here the set $\spec_{b,\N}^q(A)$, the boundary spectrum of the complex in degree $q$ at the component $\N$ of $\M$, and prove basic properties of the boundary spectrum (assuming that the boundary component $\N$ is compact), including some aspects concerning Mellin transforms of $A$-closed forms. Some of these ideas are illustrated using the $b$-de Rham complex.

Section~\ref{sUnderlyingCRcomplexes} is a systematic study of the $\deebarb$-complex of CR structures on $\N$ associated with elements of the class $\pmb \beta$. Each $\beta\in \pmb\beta$ defines a CR structure, $\Kbar_\beta=\ker \beta$. Assuming that $\Veebar$ admits a $\T$-invariant Hermitian metric, we show that there is $\beta\in \pmb \beta$ such that the CR structure $\Kbar_\beta$ is $\T$-invariant.

In Section~\ref{sSpectrum} we assume that $\Veebar$ is $\T$-invariant and show that for $\T$-invariant CR structures, a theorem proved in \cite{Me9} gives that the cohomology spaces of the associated $\deebarb$-complex, viewed as the kernel of the Kohn Laplacian at the various degrees, split into eigenspaces of $-\im\Lie_\T$. The eigenvalues of the latter operator are related to the indicial spectrum of the $\bdeebar$-complex.

In Section~\ref{sIndicialCohomology} we prove a precise theorem on the indicial cohomology and spectrum for the $\bdeebar$-complex under the assumption that $\Veebar$ admits a $\T$-invariant Hermitian metric.

Finally, we have included a very short appendix listing a number of basic definitions in connection with $b$-operators.

\section{Complex $b$-structures}\label{sPreliminaries}

Let $\M$ be a smooth manifold with smooth boundary. An almost CR $b$-structure on $\M$ is a subbundle $\Wbar$ of the complexification, $\C\bT\M\to\M$ of the $b$-tangent bundle of $\M$ (Melrose \cite{RBM1,RBM2}) such that
\begin{equation}\label{bCRstructure}
\W\cap \Wbar=0
\end{equation}
with $\W=\Wbar$. If in addition 
\begin{equation}\label{bElliptic}
\W + \Wbar=\C\bT\M
\end{equation}
then we say that $\Wbar$ is an almost complex $b$-structure and write $\bT^{0,1}\M$ instead of $\Wbar$ and $\bT^{1,0}\M$ for its conjugate. As is customary, the adverb ``almost'' is dropped if $\W$ is involutive. Note that since $C^\infty(\M;\bT\M)$ is a Lie algebra, it makes sense to speak of involutive subbundles of $\bT\M$ (or its complexification). 

\begin{definition}
A complex $b$-manifold is a manifold together with a complex $b$-structure.
\end{definition}

By the Newlander-Nirenberg Theorem~\cite{NeNi57}, the interior of complex $b$-manifold is a complex manifold. However, its boundary is not a CR manifold; rather, as we shall see, it naturally carries a family of CR structures parametrized by the defining functions of $\partial\M$ in $\M$ which are positive in $\open \M$.

\medskip
That $C^\infty(\M;\bT\M)$ is a Lie algebra is an immediate consequence of the definition of the $b$-tangent bundle, which indeed can be characterized as being a vector bundle $\bT\M\to\M$ together with a vector bundle homomorphism 
\begin{equation*}
\ev:\bT\M\to T\M
\end{equation*}
covering the identity such that the induced map
\begin{equation*}
\ev_*:C^\infty(\M;\bT\M)\to C^\infty(\M;T\M)
\end{equation*}
is a $C^\infty(\M;\R)$-module isomorphism onto the submodule $C^\infty_{\tan}(\M;T\M)$ of smooth vector fields on $\M$ which are tangential to the boundary of $\M$. Since $C^\infty_{\tan}(\M,T\M)$ is closed under Lie brackets, there is an induced Lie bracket on $C^\infty(\M;\bT\M)$ The homomorphism $\ev$ is an isomorphism over the interior of $\M$, and its restriction to the boundary,
\begin{equation}\label{evbM}
\ev_{\partial\M}:\bT_{\partial\M}\M\to T\partial\M
\end{equation}
is surjective. Its kernel, a fortiori a rank-one bundle, is spanned by a canonical section denoted $\rr\partial_\rr$. Here and elsewhere, $\rr$ refers to any smooth defining function for $\partial \M$ in $\M$, by convention positive in the interior of $\M$.

Associated with a complex $b$-structure on $\M$ there is a Dolbeault complex. Let $\bWedge^{0,q}\M$ denote the $q$-th exterior power of the dual of $\bT^{0,1}M$. Then the operator
\begin{equation*}
\cdots\to C^\infty(\M;\bWedge^{0,q}\M)\xrightarrow{\bdeebar} C^\infty(\M;\bWedge^{0,q+1}\M) \to \cdots
\end{equation*}
is define by 
\begin{multline}\label{CartanFormula}
(q+1)\,\bdeebar\phi(V_0,\dotsc,V_q)=\sum_{j=0}^q V_j\phi(V_0,\dotsc,\hat V_j,\dotsc,V_q)\\
 +\sum_{j<k}(-1)^{j+k}\phi([V_j,V_k],V_0,\dotsc,\hat V_j,\dotsc,\hat V_k,\dotsc,V_q)
\end{multline}
as with the standard de Rham differential (see Helgason~\cite[p.~21]{He}) whenever $\phi$ is a smooth section of $\bWedge^q\M$ and $V_0,\dotsc,V_q\in C^\infty(\M;\bT^{0,1}\M)$. In this formula $V_j$ acts on functions via the vector field $\ev_*V_j$. The involutivity of $\bT^{0,1}\M$ is used in the the terms involving brackets, of course. The same proof that $d\circ d=0$ works here to give that $\bdeebar^2=0$. The formula
\begin{equation}
\bdeebar(f\phi)=f\,\bdeebar\phi+\bdeebar f\wedge \phi\text{ for } \phi\in C^\infty(\M;\bWedge^q\M) \text{ and } f\in C^\infty(\M),
\end{equation}
implies that $\bdeebar$ is a  first order operator.

Since we do not have at our disposal holomorphic frames (near the boundary) for the bundles of forms of type $(p,q)$ for $p>0$, we define $\bdeebar$ on forms of type $(p,q)$ with $p>0$ with the aid of the $b$-de Rham complex, exactly as in Foland and Kohn~\cite{FK} for standard complex structures and de Rham complex. The $b$-de Rham complex, we recall from Melrose~\cite{RBM2}, is the complex associated with the dual, $\C\bT^*\M$, of $\C\bT\M$,
\begin{equation*}
\cdots\to C^\infty(\M;\bWedge^r\M)\xrightarrow{\bd} C^\infty(\M;\bWedge^{r+1}\M) \to \cdots
\end{equation*}
where $\bWedge^q\M$ denotes the $r$-th exterior power of $\C\bT^*\M$. The operators $\bd$ are defined by the same formula as \eqref{CartanFormula}, now however with the $V_j\in C^\infty(\M;\C\bT\M)$. On functions $f$ we have 
\begin{equation*}
\bd f=\ev^*df.
\end{equation*}
More generally,
\begin{equation*}
\ev^* \circ d = \bd\circ \ev^*
\end{equation*}
in any degree. Also,
\begin{equation}\label{FirstOrder}
\bd(f\phi)=f\,\bd\phi+\bd f\wedge \phi\text{ for } \phi\in C^\infty(\M;\bWedge^r\M) \text{ and } f\in C^\infty(\M).
\end{equation}
It is convenient to note here that for $f\in C^\infty(\M)$,
\begin{equation}\label{VanishingOnBdy}
\text{$\bd f$ vanishes on $\partial \M$ if $f$ does.}
\end{equation}

Now, with the obvious definition,
\begin{equation}\label{SpittingOfDeRham}
\bWedge^r\M=\bigoplus_{p+q=r}\bWedge^{p,q}\M.
\end{equation}
Using the special cases 
\begin{gather*}
 \bd:C^\infty(\M;\bWedge^{0,1})\to C^\infty(\M;\bWedge^{1,1})+C^\infty(\M;\bWedge^{0,2}),\\
 \bd:C^\infty(\M;\bWedge^{1,0})\to C^\infty(\M;\bWedge^{2,0})+C^\infty(\M;\bWedge^{1,1}),
\end{gather*}
consequences of the involutivity of $\bT^{0,1}\M$ and its conjugate, one gets 
\begin{equation*}
\bd\phi\in C^\infty(\M;\bWedge^{p+1,q}\M)\oplus C^\infty(\M;\bWedge^{p,q+1}\M)\quad \text{if } \phi\in C^\infty(\M;\bWedge^{p,q}\M)
\end{equation*}
for general $(p,q)$. Let $\pi_{p,q}:\bWedge^k\M\to\bWedge^{p,q}\M$ be the projection according to the decomposition \eqref{SpittingOfDeRham}, and define
\begin{equation*}
\bdee=\pi_{p+1,q}\bd,\quad \bdeebar=\pi_{q,p+1}\bd,
\end{equation*}
so $\bd=\bdee+\bdeebar$. The operators $\bdeebar$ are identical to the $\deebar$-operators over the interior of $\M$ and with the previously defined $\bdeebar$ operators on $(0,q)$-forms, and give a complex
\begin{equation}\label{bdeebarComplex}
\cdots
\to C^\infty(\M;\bWedge^{p,q}\M) \xrightarrow{\bdeebar}
C^\infty(\M;\bWedge^{p,q+1}\M)
\to\cdots
\end{equation}
for each $p$.  On functions $f:\M\to\C$,
\begin{equation}\label{bdeebarOnFunctions}
\bdeebar f = \pi_{0,1}\, \bd f.
\end{equation}
The formula
\begin{equation}\label{bdeebarFirstOrderPrime}\tag{\ref{FirstOrder}$'$}
\bdeebar f\phi=\bdeebar f\wedge \phi+f\bdeebar\phi ,\quad f\in C^\infty(\M),\ \phi\in C^\infty(\M;\bWedge^{p,q}\M),
\end{equation}
a consequence of \eqref{FirstOrder}, implies that $\bdeebar$ is a first order operator. As a consequence of \eqref{VanishingOnBdy}, 
\begin{equation}\label{bdeebarVanishingOnBdy}\tag{\ref{VanishingOnBdy}$'$}
\text{$\bdeebar f$ vanishes on $\partial \M$ if $f$ does.}
\end{equation}

\medskip
The operators of the $b$-de Rham complex are first order operators because of \eqref{FirstOrder}, and \eqref{VanishingOnBdy} implies that these are $b$-operators, see \eqref{TotallyChar}. Likewise, \eqref{bdeebarFirstOrderPrime} and \eqref{bdeebarVanishingOnBdy} imply that in any bidegree, the operator $\phi\mapsto \rr^{-1}\,\bdeebar \,\rr\phi$ has coefficients smooth up to the boundary, so
\begin{equation}\label{bdeebarOnpq}
\bdeebar\in \Diff^1_b(\M;\bWedge^{p,q}\M,\bWedge^{p,q+1}\M),
\end{equation}
see \eqref{TotallyChar}. We also get from these formulas that the $b$-symbol of $\bdeebar$ is
\begin{equation}\label{bsymbdeebar}
\bsym(\bdeebar)(\xi)(\phi)=\im \pi_{0,1}(\xi)\wedge \phi, \quad x\in\M,\ \xi\in \bT^*_x\M,\ \phi\in \bWedge^{p,q}_x\M,
\end{equation}
see \eqref{TheBSymbol}. Since $\pi_{0,1}$ is injective on the real $b$-cotangent bundle (this follows from \eqref{bElliptic}), the complex \eqref{bdeebarComplex} is $b$-elliptic.

\section{Holomorphic vector bundles}\label{sHolomorphicVectorBundles}
The notion of holomorphic vector bundle in the $b$-category is a translation of the standard one using connections. Let $\rho:F\to\M$ be a complex vector bundle. Recall from \cite{RBM2} that a $b$-connection on $F$ is a linear operator
\begin{equation*}
\bnabla:C^\infty(\M;F)\to C^\infty(\M;\bWedge^1\M \otimes F)
\end{equation*}
such that
\begin{equation}\label{Connection}
\bnabla f\phi=f\, \bnabla\phi+ \bd f \otimes \phi
\end{equation}
for each $\phi \in C^\infty(\M;F)$ and $f\in C^\infty(\M)$. This property automatically makes $\bnabla$ a $b$-operator.

A standard connection $\nabla:C^\infty(\M;F)\to C^\infty(\M;\Wedge^1\M \otimes F)$ determines a $b$-connection by composition with
\begin{equation*}
\ev^*\otimes\Id:\Wedge^1\M \otimes F\to \bWedge^1\M \otimes F,
\end{equation*}
but $b$-connections are more general than standard connections. Indeed, the difference between the latter and the former can be any smooth section of the bundle $\Hom(F,\bWedge^1\M\otimes F)$. A $b$-connection $\bnabla$ on $F$ arises from a standard connection if and only if $\bnabla_{\rr\partial_\rr}=0$ along $\partial\M$.

As in the standard situation, the $b$-connection $\bnabla$ determines operators
\begin{equation}\label{ExtConnection}
\bnabla:C^\infty(\M;\bWedge^k\M \otimes F)\to C^\infty(\M;\bWedge^{k+1}\M \otimes F)
\end{equation}
by way of the usual formula translated to the $b$ setting:
\begin{equation}\label{ExtConnectionBis}
\bnabla(\alpha\otimes \phi) = (-1)^k \alpha \wedge \bnabla \phi + \bd\alpha \wedge \phi,\quad \phi \in C^\infty(\M;F),\ \alpha\in \bWedge^k\M.
\end{equation}
Since
\begin{equation*}
\bnabla \rr\alpha\otimes \phi=\rr\, \bnabla(\alpha\otimes \phi) + \bd \rr \wedge \alpha\otimes \phi
\end{equation*}
is smooth and vanishes on $\partial\M$, also
\begin{equation*}
\bnabla\in \Diff^1_b(\M;\bWedge^k\M\otimes F,\bWedge^{k+1}\M\otimes F).
\end{equation*}
The principal $b$-symbol of \eqref{ExtConnection}, easily computed using \eqref{ExtConnectionBis} and
\begin{equation*}
\bsym(\bnabla)(\bd f)(\phi) = \lim_{\tau\to\infty} \frac{e^{-\im\tau f}}{\tau} \bnabla e^{\im\tau f}\phi
\end{equation*}
for $f\in C^\infty(\M;\R)$ and $\phi \in C^\infty(\M;\bWedge^k\M \otimes F)$, is
\begin{equation*}
\bsym(\bnabla)(\xi)(\phi) = \im\xi\wedge \phi,\quad \xi\in \bT_x^*\M,\ \phi \in \bWedge^k_x\M \otimes F_x,\ x\in \M.
\end{equation*}

As expected, the connection is called holomorphic if the component in $\bWedge^{0,2}\M\otimes F$ of the curvature operator
\begin{equation*}
\Omega=\bnabla^2:C^\infty(\M;F)\to C^\infty(\M;\bWedge^2\M\otimes F),
\end{equation*}
vanishes. Such a connection gives $F$ the structure of a complex $b$-manifold. Its complex $b$-structure can be described locally as in the standard situation, as follows. Fix a frame $\eta_\mu$ for $F$ and let the $\omega^\nu_\mu$ be the local sections of $\bWedge^{0,1}\M$ such that
\begin{equation*}
\bdeebar \eta_\mu = \sum_\nu \omega^\nu_\mu\otimes \eta_\nu.
\end{equation*}
Denote by $\zeta^\mu$ the fiber coordinates determined by the frame $\eta_\mu$. Let $V_1,\dotsc, V_{n+1}$ be a frame of $\bT^{0,1}\M$ over $U$, denote by $\tilde V_j$ the sections of $\C\bT F$ over $\rho^{-1}(U)$ which project on the $V_j$ and satisfy $\tilde V_j\zeta^\mu=\tilde V_j\overline \zeta^\mu=0$ for all $\mu$, and by $\partial_{\zeta^\mu}$ the vertical vector fields such that $\partial_{\zeta^\mu}\zeta^\nu=\delta^\nu_\mu$ and $\partial_{\zeta^\mu}\overline \zeta^\nu=0$. Then the sections
\begin{equation}\label{LocalbT01E}
\tilde V_j-\sum_{\mu,\nu}\zeta^\mu \langle \omega^\nu_\mu, V_j\rangle \partial_{\zeta^\nu},\ j=1,\dotsc,n+1,\quad \partial_{\overline \zeta^\nu},\ \nu=1,\dotsc,k
\end{equation}
of $\C\bT F$ over $\rho^{-1}(U)$ form a frame of $\bT^{0,1}F$. As in the standard situation, the involutivity of this subbundle of $\C\bT F$ is equivalent to the condition on the vanishing of the $(0,2)$ component of the curvature of $\bnabla$. A vector bundle $F\to\M$ together with the complex $b$-structure determined by a choice of holomorphic $b$-connection (if one exists at all) is a holomorphic vector bundle.

The $\deebar$ operator of a holomorphic vector bundle is
\begin{equation*}
\bdeebar = (\pi_{0,q+1}\otimes \Id)\circ \bnabla : C^\infty(\M;\bWedge^{0,q}\M\otimes F)\to C^\infty(\M;\bWedge^{0,q+1}\M \otimes F).
\end{equation*}
As is the case for standard complex structures, the condition on the curvature of $\bnabla$ implies that these operators form a complex, $b$-elliptic since
\begin{equation*}
\bsym(\bdeebar)(\xi)(\phi) = \im\pi_{0,1}(\xi)\wedge \phi,\quad \xi\in \bT_x^*\M,\ \phi \in \bWedge^k_x\M \otimes F_x,\ x\in \M
\end{equation*}
and $\pi_{0,1}(\xi)=0$ for $\xi\in \bT^*\M$ if and only if $\xi=0$.

Also as usual, a $b$-connection $\bnabla$ on a Hermitian vector bundle $F\to\M$ with Hermitian form $h$ is Hermitian if
\begin{equation*}
\bd h(\phi,\psi)= h(\bnabla \phi,\psi)+h(\phi,\bnabla\psi)
\end{equation*}
for every pair of smooth sections $\phi$, $\psi$ of $F$. In view of the definition of $\bd$ this means that for every $v\in \C\bT\M$ and sections as above,
\begin{equation*}
\ev(v) h(\phi,\psi)= h(\bnabla_{\!v} \phi,\psi)+h(\phi,\bnabla_{\!\overline v}\psi)
\end{equation*}
On a complex $b$-manifold $\M$, if an arbitrary connection $\bnabla'$ and the Hermitian form $h$ are given for a vector bundle $F$, holomorphic or not, then there is a unique \emph{Hermitian} $b$-connection $\bnabla$ such that $\pi_{0,1}\bnabla = \pi_{0,1}\bnabla'$. Namely, let $\eta_\mu$ be a local orthonormal frame of $F$, let
\begin{equation*}
(\pi_{0,1}\otimes\Id)\circ \bnabla'\eta_\mu = \sum_\nu\omega^\nu_\mu \otimes \eta_\nu,
\end{equation*}
and let $\bnabla$ be the connection defined in the domain of the frame by
\begin{equation}\label{HermitianConnectionForms}
\bnabla\eta_\mu =(\omega^\nu_\mu-\overline \omega^\mu_\nu)\otimes \eta_\nu.
\end{equation}
If the matrix of functions $Q=[q^\mu_\lambda]$ is unitary and $\tilde \eta_\lambda=\sum_\mu q^\mu_\lambda \eta_\mu$, then 
\begin{equation*}
(\pi_{0,1}\otimes\Id)\circ \bnabla'\tilde \eta_\lambda = \sum_\nu\tilde \omega^\sigma_\lambda \otimes \tilde \eta_\sigma
\end{equation*}
with
\begin{equation*}
\tilde \omega^\sigma_\lambda = \sum_\mu \overline q^\mu_\sigma\,\bdeebar q^\mu_\lambda + \sum_{\mu,\nu} \overline q^\mu_\sigma q^\nu_\lambda\omega^\mu_\nu,
\end{equation*}
using \eqref{Connection}, that $Q^{-1}=[\overline q^\mu_\lambda]$, and that $\pi_{0,1}\bd f=\bdeebar f$. Thus
\begin{align*}
\tilde \omega^\sigma_\lambda-\overline {\tilde \omega}^\lambda_\sigma
&= \sum_\mu (\overline q^\mu_\sigma\,\bdeebar q^\mu_\lambda -  q^\mu_\lambda\,\bdee \overline q^\mu_\sigma)  + \sum_{\mu,\nu} (\overline q^\mu_\sigma q^\nu_\lambda\omega^\mu_\nu -  q^\mu_\lambda \overline q^\nu_\sigma\overline \omega^\mu_\nu)\\
&=\sum_\mu (\bdeebar q^\mu_\lambda + \bdee q^\mu_\lambda)\overline q^\mu_\sigma  + \sum_{\mu,\nu} q^\nu_\lambda ( \omega^\mu_\nu - \overline \omega^\nu_\mu)\overline q^\mu_\sigma\\
&=\sum_\mu \bd q^\mu_\lambda + \overline q^\mu_\sigma  + \sum_{\mu,\nu} q^\nu_\lambda ( \omega^\mu_\nu - \overline \omega^\nu_\mu)\overline q^\mu_\sigma
\end{align*}
using that $\overline{\bdeebar f}=\bdee \overline f$ and that $\sum_\mu  q^\mu_\lambda\,\bdee \overline q^\mu_\sigma = - \sum_\mu \bdee q^\mu_\lambda\, \overline q^\mu_\sigma$ because $\sum_\mu q^\mu_\lambda \overline q^\mu_\sigma$ is constant, and that $\bdeebar q^\mu_\lambda + \bdee q^\mu_\lambda=\bd q^\mu_\lambda$. Thus there is a globally defined Hermitian connection locally given by \eqref{HermitianConnectionForms}. We leave to the reader to verify that this connection is Hermitian. Clearly $\bnabla$ is the unique Hermitian connection such that $\pi_{0,1}\bnabla = \pi_{0,1}\bnabla'$. When $\bnabla'$ is a holomorphic connection, $\bnabla$ is the unique Hermitian holomorphic connection.

\begin{lemma}
The vector bundles $\bWedge^{p,0}\M$ are holomorphic.
\end{lemma}

We prove this by exhibiting a holomorphic $b$-connection. Fix an auxiliary Hermitian metric on $\bWedge^{p,0}\M$ and pick an orthonormal frame $(\eta_\mu)$ of $\bWedge^{p,0}\M$ over some open set $U$. Let $\omega^\nu_\mu$ be the unique sections of $\bWedge^{0,1}\M$ such that
\begin{equation*}
\bdeebar\eta_\mu =\sum_\nu \omega^\nu_\mu\wedge\eta_\nu,
\end{equation*}
and let $\bnabla$ be the $b$-connection defined on $U$ by the formula \eqref{HermitianConnectionForms}. As in the previous paragraph, this gives a globally defined $b$-connection. That it is holomorphic follows from
\begin{equation*}
\bdeebar \omega^\nu_\mu + \sum_\lambda \omega^\nu_\lambda\wedge \omega^\lambda_\mu = 0,
\end{equation*}
a consequence of $\bdeebar^2=0$. Evidently, with the identifications $\bWedge^{0,q}\M\otimes \bWedge^{p,0}\M=\bWedge^{p,q}\M$, $\pi_{p,q+1}\bnabla$ is the $\bdeebar$ operator in \eqref{bdeebarOnpq}.

\section{The boundary a complex $b$-manifold}\label{sBoundaryStructure}

Suppose that $\M$ is a complex $b$-manifold and $\N$ is a component of its boundary. We shall assume $\N$ compact, although for the most part this is not necessary.

The homomorphism
\begin{equation*}
\ev:\C\bT\M\to \C T\M
\end{equation*}
is an isomorphism over the interior of $\M$, and its restriction to $\N$ maps onto $\C T\N$ with kernel spanned by $\rr\partial_\rr$. Write
\begin{equation*}
\ev_\N:\C\bT_{\N}\M\to \C T\N
\end{equation*}
for this restriction and 
\begin{equation}\label{bdyIso}
\Phi:\bT^{0,1}_{\N}\M\to\Veebar
\end{equation}
for of the restriction of $\ev_\N$ to $\bT^{0,1}_{\N}\M$. From \eqref{bCRstructure} and the fact that the kernel of $\ev_\N$ is spanned by the real section $\rr\partial_{\rr}$ one obtains that $\Phi$ is injective, so its image,
\begin{equation*}
\Veebar=\Phi(\bT^{0,1}_{\N}\M)
\end{equation*}
is a subbundle of $\C T\N$.

Since $\bT^{0,1}\M$ is involutive, so is $\Veebar$, see \cite[Proposition 3.12]{Me3}. From \eqref{bElliptic} and the fact that $\ev_\N$ maps onto $\C T\N$, one obtains that
\begin{equation}\label{VbarIsElliptic}
\Vee + \Veebar=\C T\N, 
\end{equation}
see \cite[Lemma 3.13]{Me3}. Thus

\begin{lemma}
$\Veebar$ is an elliptic structure.
\end{lemma}

This just means what we just said: $\Veebar$ is involutive and \eqref{VbarIsElliptic} holds, see Treves~\cite{Tr81,Tr92}; the sum need not be direct. All elliptic structures are locally of the same kind, depending only on the dimension of $\Vee\cap \Veebar$. This is a result of Nirenberg~\cite{Ni57} (see also H\"ormander~\cite{Ho65}) extending the Newlander-Nirenberg theorem. In the case at hand, $\Veebar\cap \Vee$ has rank $1$ because of the relation
\begin{equation*}
\rk_\C(\Vee\cap \Veebar)=2\rk_\C\Veebar-\dim\N
\end{equation*}
which holds whenever \eqref{VbarIsElliptic} holds. 
\begin{equation}\label{HypoanalyticChart}
\display{300pt}{Every $p_0\in \N$ hs a neighborhood in which there coordinates $x^1,\dotsc,x^{2n},t$ such that with $z^j=x^j+\m x^{j+n}$, the vector fields 
\begin{equation*}
\hspace*{-60pt}\frac{\partial}{\partial \overline z^1},\dotsc,\frac{\partial}{\partial \overline z^n},\ \frac{\partial }{\partial t}
\end{equation*}
span $\Veebar$ near $p_0$. The function $(z^1,\dotsc,z^n,t)$ is called a hypoanalytic chart (Baouendi, Chang, and Treves~\cite{BCT83}, Treves~\cite{Tr92}).}
\end{equation}

The intersection $\Veebar\cap \Vee$ is, in the case we are discussing, spanned by a canonical globally defined real vector field.
Namely, let $\rr\partial_\rr$ be the canonical section of $\bT\M$ along $\N$. There is a unique section $J\rr\partial_\rr$ of $\bT\M$ along $\N$ such that $\rr\partial_\rr+\im J\rr \partial_\rr$ is a section of $\bT^{0,1}\M$ along $\N$. Then
\begin{equation*}
\T=\ev_\N(J\rr\partial_\rr)
\end{equation*}
is a nonvanishing real vector field in $\Vee\cap \Veebar$, (see \cite[Lemma 2.1]{Me4}). Using the isomorphism \eqref{bdyIso} we have
\begin{equation*}
\T=\Phi(J(\rr\partial_\rr)-\im\rr\partial_\rr).
\end{equation*}

Because $\Veebar$ is involutive, there is yet another complex, this time associated with the exterior powers of the dual of $\Veebar$:
\begin{equation}\label{bdyComplex}
\cdots
\to C^\infty(\N;\Wedge^q \smash[t]{\Veebar}^*)\xrightarrow{\Deebar}
C^\infty(\N;\Wedge^{q+1}\smash[t]{\Veebar}^*)
\to\cdots,
\end{equation}
where $\Deebar$ is defined by the formula \eqref{CartanFormula} where now the $V_j$ are sections of $\Veebar$. The complex \eqref{bdyComplex} is elliptic because of \eqref{VbarIsElliptic}. For a function $f$ we have $\Deebar f=\iota^*df$, where $\iota^*:\C T^*\N\to \smash[t]{\Veebar}^*$ is the dual of the inclusion homomorphism $\iota:\Veebar\to\C T\N$.

For later use we show:
\begin{lemma}\label{ConstantSolutions}
Suppose that $\N$ is compact and connected. If $\zeta:\N\to\C$ solves $\Deebar\zeta=0$, then $\zeta$ is constant.
\end{lemma}

\begin{proof}
Let $p_0$ be an extremal point of $|\zeta|$. Fix a hypoanalytic chart $(z,t)$ for $\Veebar$ centered at $p_0$. Since $\Deebar \zeta=0$, $\zeta(z,t)$ is independent of $t$ and $\partial_{\overline z^\nu}\zeta=0$. So there is a holomorphic function $Z$ defined in a neighborhood of $0$ in $\C^n$ such that $\zeta=Z\circ z$. Then $|Z|$ has a maximum at $0$, so $Z$ is constant near $0$. Therefore $\zeta$ is constant, say $\zeta(p)=c$, near $p_0$. Let $C=\set{p:\zeta(p)=c}$, a closed set. Let $p_1\in C$. Since $p_1$ is also an extremal point of $\zeta$, the above argument gives that $\zeta$ is constant near $p_1$, therefore equal to $c$. Thus $C$ is open, and consequently $\zeta$ is constant on $\N$.
\end{proof}

Since the operators $\bdeebar:C^\infty(\M,\bWedge^{0,q}\M)\to C^\infty(\M,\bWedge^{0,q+1}\M)$ are totally characteristic, they induce operators
\begin{equation*}
\bdeebar_b:C^\infty(\N,\bWedge^{0,q}_\N\M)\to C^\infty(\M,\bWedge^{0,q+1}_\N\M),
\end{equation*}
see \eqref{Pb}; these boundary operators define a complex because of \eqref{PQb}. By way of the dual
\begin{equation}\label{DualbdyIso}
\Phi^*:\smash[t]{\Veebar}^*\to \bWedge^{0,1}_\N\M
\end{equation}
of the isomorphism \eqref{bdyIso} the operators $\bdeebar_b$ become identical to the operators of the $\Deebar$-complex \eqref{bdyComplex}: The diagram
\begin{equation*}
\begin{CD}
\cdots &@>>> C^\infty(\N;\Wedge^q \smash[t]{\Veebar}^*)&@>\Deebar>>C^\infty(\N;\Wedge^{q+1}\smash[t]{\Veebar}^*)&@>>>&\cdots\\
&& &@V{\Phi^*}VV&@VV{\Phi^*}V&\\
\cdots &@>>> C^\infty(\N,\bWedge^{0,q}_\N\M)&@>\bdeebar_b>>C^\infty(\M,\bWedge^{0,q+1}_\N\M)&@>>>&\cdots
\end{CD}
\end{equation*}
is commutative and the vertical arrows are isomorphisms. This can be proved by writing the $\bdeebar$ operators using Cartan's formula \eqref{CartanFormula} for $\bdeebar$ and $\Deebar$ and comparing the resulting expressions.

Let $\rr:\M\to\R$ be a smooth defining function for $\partial \M$, $\rr>0$ in the interior of $\M$. Then $\bdeebar \rr$ is smooth and vanishes on $\partial \M$, so $\frac{\bdeebar \rr}{\rr}$ is also a smooth $\bdeebar$-closed section of $\bWedge^{0,1}\M$. Thus we get a $\Deebar$-closed element
\begin{equation}\label{DefinitionBeta}
\beta_\rr=[\Phi^*]^{-1}\frac{\bdeebar \rr}{\rr} \in C^\infty(\partial \M;\smash[t]{\Veebar}^*).
\end{equation}
By definition,
\begin{equation*}
\langle \beta_\rr,\T\rangle=\langle \frac{\bdeebar \rr}{\rr},J(\rr\partial_\rr)-\im\rr\partial_\rr\rangle.
\end{equation*}
Extend the section $\rr\partial_\rr$ to a section of $\bT\M$ over a neighborhood $U$ of $\N$ in $\M$ with the property that $\rr\partial_\rr\rr=\rr$. In $U$ we have
\begin{equation*}
\langle \bdeebar \rr,J(\rr\partial_\rr)-\im\rr\partial_\rr\rangle = (J(\rr\partial_\rr)-\im\rr\partial_\rr)\rr = J(\rr\partial_\rr)\rr -\im\rr.
\end{equation*}
The function $J(\rr\partial_\rr)\rr$ is smooth, real-valued, and vanishes along the boundary. So $\rr^{-1}J(\rr\partial_\rr)\rr$ is smooth, real-valued. Thus
\begin{equation*}
\langle\beta_\rr,\T\rangle = a_\rr-\im
\end{equation*}
on $\N$ for some smooth function $a_\rr:\N\to\R$, see \cite[Lemma 2.5]{Me4}.

If $\rr'$ is another defining function for $\partial \M$, then $\rr'=\rr e^u$ for some smooth function $u:\M\to\R$. Then
\begin{equation*}
\bdeebar \rr'=e^u\,\bdeebar \rr+ e^u\rr\,\bdeebar u
\end{equation*}
and it follows that
\begin{equation*}
\beta_{\rr'}=\beta_\rr+ \Deebar u.
\end{equation*}
In particular,
\begin{equation*}
a_{\rr'}=a_\rr+ \T u.
\end{equation*}
Let $\a_t$ denote the one-parameter group of diffeomorphisms generated by $\T$.

\begin{proposition}\label{Averages}
The functions $a^{\sup}_\av$, $a^{\inf}_\av:\N\to \R$ defined by
\begin{equation*}
a^{\sup}_\av(p) =\limsup_{t\to\infty}\frac{1}{2t}\int_{-t}^t a_{\rr}(\a_s(p))\,ds,\quad a^{\inf}_\av(p) =\liminf_{t\to\infty}\frac{1}{2t}\int_{-t}^t a_{\rr}(\a_s(p))\,ds
\end{equation*}
are invariants of the complex $b$-structure, that is, they are independent of the defining function $\rr$. The equality $a^{\sup}_\av = a^{\inf}_\av$ holds for some $\rr$ if and only if it holds for all $\rr$. 
\end{proposition}

Indeed,
\begin{equation*}
\lim_{t\to\infty}\Big( \frac{1}{2t}\int_{-t}^t a_{\rr'}(\a_s(p))\,ds - \frac{1}{2t}\int_{-t}^t a_{\rr}(\a_s(p))\,ds\Big) = \lim_{t\to\infty}\frac{1}{2t}\int_{-t}^t \frac{d}{ds}u(\a_s(p))\,ds = 0
\end{equation*}
because $u$ is bounded (since $\N$ is compact).

The functions $a^{\sup}_\av$, $a^{\inf}_\av$ are constant on orbits of $\T$, but they may not be smooth.

\begin{example}\label{AnisotropicSphere}
Let $\N$ be the unit sphere in $\C^{n+1}$ centered at the origin. Write $(z^1,\dotsc,z^{n+1})$ for the standard coordinates in $\C^{n+1}$. Fix $\tau_1,\dotsc,\tau_{n+1}\in \R\minus 0$, all of the same sign, and let
\begin{equation*}
\T= \im\sum_{j=1}^{n+1} \tau_j(z^j\partial_{z^j}-\overline z^j\partial_{\overline z^j}).
\end{equation*}
This vector field is real and tangent to $\N$. Let $\Kbar$ be the standard CR structure of $\N$ as a submanifold of $\C^{n+1}$ (the part of $T^{0,1}\C^{n+1}$ tangential to $\N$). The condition that the $\tau_j$ are different from $0$ and have the same sign ensures that $\T$ is never in $\K\oplus \Kbar$. Indeed, the latter subbundle of $\C T\N$ is the annihilator of the pullback to $\N$ of $\im \deebar \sum_{\ell=1}^{n+1}|z^\ell|^2$. The pairing of this form with $\T$ is
\begin{equation*}
\langle \im \sum_{\ell=1}^{n+1} z^\ell d\overline z^\ell,\im\sum_{j=1}^{n+1} \tau_j(z^j\partial_{z^j}-\overline z^j\partial_{\overline z^j})\rangle = \sum_{j=1}^{n+1} \tau_j|z^j|^2,
\end{equation*}
a function that vanishes nowhere if and only if all $\tau_j$ are different from zero and have the same sign. Thus $\Veebar=\Kbar\oplus \Span_\C\T$ is a subbundle of $\C T\N$ of rank $n+1$ with the property that $\Vee+\Veebar=\C T\N$. To show that $\Veebar$ is involutive we first note that $\Kbar$ is the annihilator of the pullback to $\N$ of the span of the differentials $dz^1,\dotsc,dz^{n+1}$. Let $\Lie_\T$ denote the Lie derivative with respect to $\T$. Then $\Lie_\T dz^j =\im \tau_j dz^j$, so if $L$ is a CR vector field, then so is $[L,\T]$. Since in addition $\Kbar$ and $\Span_\C \T$ are themselves involutive, $\Veebar$ is involutive. Thus $\Veebar$ is an elliptic structure with $\Vee\cap \Veebar=\Span_\C\T$. Let $\beta$ be the section of $\smash[t]{\Veebar}^*$ which vanishes on $\Kbar$ and satisfies $\langle \beta,\T\rangle=-\im$. Let $\Deebar$ denote the operators of the associated differential complex. Then $\Deebar\beta=0$, since $\beta$ vanishes on commutators of sections of $\Kbar$ (since $\Kbar$ is involutive) and on commutators of $\T$ with sections of $\Kbar$ (since such commutators are in $\Kbar$).

If the $\tau_j$ are positive (negative), this example may be viewed as the boundary of a blowup (compactification) of $\C^{n+1}$, see \cite{Me5}.
\end{example}

\medskip
Let now $\rho:F\to\M$ be a holomorphic vector bundle. Its $\bdeebar$-complex also determines a complex along $\N$,
\begin{equation}\label{defDeebarBundle}
\cdots \to C^\infty(\N;\Wedge^q\smash[t]{\Veebar}^*\otimes F_\N)\xrightarrow{\Deebar} C^\infty(\N;\Wedge^{q+1}\smash[t]{\Veebar}^*\otimes F_\N)\to\cdots,
\end{equation}
where $\Deebar$ is defined using the boundary operators $\bdeebar_b$ and the isomorphism \eqref{DualbdyIso}:
\begin{equation}\label{defDeebarE}
\Deebar (\phi\otimes \eta) = (\Phi^*)^{-1}\bdeebar_b [\Phi^* (\phi\otimes \eta)]
\end{equation}
where $\Phi^*$ means $\Phi^*\otimes\Id$. These operators can be expressed locally in terms of the operators of the complex \eqref{bdyComplex}. Fix a smooth frame $\eta_\mu$, $\mu=1,\dotsc,k$, of $F$ in a neighborhood $U\subset \M$ of $p_0\in \N$, and suppose
\begin{equation*}
\bdeebar \eta_\mu = \sum_\nu\omega^\nu_\mu\otimes \eta_\nu.
\end{equation*}
The $\omega^\nu_\mu$ are local sections of $\bWedge^{0,1}\M$, and if $\sum_\mu \phi^\mu\otimes \eta_\mu$ is a section of $\bWedge^{0,q}\M\otimes F$ over $U$, then
\begin{equation*}
\bdeebar\sum\phi^\mu\otimes\eta_\mu = \sum_\nu (\bdeebar \phi^\nu +\sum_\mu \omega^\nu_\mu\wedge \phi^\mu)\otimes \eta_\nu.
\end{equation*}
Therefore, using the identification \eqref{DualbdyIso}, the boundary operator $\bdeebar_b$ is the operator given locally by
\begin{equation}\label{DeebarE}
\Deebar\sum\phi^\mu\otimes\eta_\mu = \sum_\nu (\Deebar \phi^\nu +\sum_\mu \omega^\nu_\mu\wedge \phi^\mu)\otimes \eta_\nu
\end{equation}
where now the $\phi^\mu$ are sections of $\bWedge^q\smash[t]{\Veebar}^*$, the $\omega^\nu_\mu$ are the sections of $\smash[t]{\Veebar}^*$ corresponding to the original $\omega^\nu_\mu$ via $\Phi^*$, and $\Deebar$ on the right hand side of the formula is the operator associated with $\Veebar$.

The structure bundle $\bT^{0,1} F$ is locally given as the span of the sections \eqref{LocalbT01E}. Applying the evaluation homomorphism $\C\bT_{\partial F} F\to \C T\partial F$ (over $\N$) to these sections gives vector fields on $F_\N$ forming a frame for the elliptic structure $\Veebar_F$ inherited by $F_\N$. Writing $V_j^0=\ev V_j$, this frame is just
\begin{equation}\label{LocalVbarE}
\tilde V_j^0-\sum_{\mu,\nu}\zeta^\mu \langle \omega^\nu_\mu, V_j^0\rangle \partial_{\zeta^\nu},\ j=1,\dotsc,n+1,\quad \partial_{\overline \zeta^\nu},\ \nu=1,\dotsc,k,
\end{equation}
where now the $\omega^\nu_\mu$ are the forms associated to the $\Deebar$ operator of $F_\N$. Alternatively, one may take the $\Deebar$ operators of $F_\N$ and use the formula \eqref{DeebarE} to define a subbundle of $\C TF$ locally as the span of the vector fields \eqref{LocalVbarE}, a fortiori an elliptic structure on $F_\N$, involutive because
\begin{equation*}
\Deebar\omega^\nu + \sum_\lambda \omega^\nu_\lambda\wedge \omega^\lambda_\mu = 0.
\end{equation*}

To obtain a formula for the canonical real vector field $\T_F$ in $\Veebar_F$, let $J_F$ be the almost complex $b$-structure of $\bT F$ and consider again the sections \eqref{LocalbT01E}; they are defined in an open set $\rho^{-1}(U)$, $U$ a neighborhood in $\M$ of a point of $\N$. Since the elements $\partial_{\overline\zeta^\nu}$ are sections of $\bT^{0,1}F$,
\begin{equation}\label{JEVert}
J_F\Re\partial_{\overline\zeta^\nu}=\Im\partial_{\overline\zeta^\nu}.
\end{equation}
Pick a defining function $\rr$ for $\N$. Then $\tilde\rr=\rho^*\rr$ is a defining function for $F_\N$. We may take $V_{n+1}=\rr\partial_\rr+\im J\rr\partial_\rr$ along $U\cap \N$. Then $\tilde V_{n+1}=\tilde \rr\partial_{\tilde \rr}+\im \widetilde {J\rr \partial_\rr}$ along $\rho^{-1}(U)\cap F_\N$ and so
\begin{multline*}
J_F\Re\big(\tilde \rr\partial_{\tilde \rr} + \im \widetilde{J\rr\partial_\rr}-\sum_{\mu,\nu} \zeta^\mu\langle\omega^\nu_\mu,\rr\partial_\rr+\im J\rr\partial_\rr\rangle\partial_{\zeta^\nu}\big) =\\ \Im\big(\tilde \rr\partial_{\tilde \rr} + \im \widetilde{J\rr\partial_\rr}-\sum_{\mu,\nu} \zeta^\mu\langle\omega^\nu_\mu,\rr\partial_\rr+\im J\rr\partial_\rr\rangle\partial_{\zeta^\nu}\big)
\end{multline*}
along $\rho^{-1}(U)\cap F_\N$. Using \eqref{JEVert} this gives
\begin{equation*}
J_F\tilde \rr\partial_{\tilde \rr} =
\widetilde{J\rr\partial_\rr} - 2\Im\sum _{\mu,\nu} \zeta^\mu\langle\omega^\nu_\mu,\rr\partial_\rr+\im J\rr\partial_\rr\rangle\partial_{\zeta^\nu}.
\end{equation*}
Applying the evaluation homomorphism gives
\begin{equation}\label{TE}
\T_F =
\tilde\T - 2\Im\sum _{\mu,\nu} \zeta^\mu\langle\omega^\nu_\mu,\rr\partial_\rr+\im J\rr\partial_\rr\rangle\partial_{\zeta^\nu}
\end{equation}
where $\tilde \T$ is the real vector field on $\rho^{-1}(U\cap \N)=\rho^{-1}(U)\cap F_\N$ which projects on $\T$ and satisfies $\tilde \T\zeta^\mu=0$ for all $\mu$.

Let $h$ be a Hermitian metric on $F$, and suppose that the frame $\eta_\mu$ is orthonormal. Applying $\T_E$ as given in \eqref{TE} to the function $|\zeta|^2=\sum|\zeta^\mu|^2$ we get that $\T_F$ is tangent to the unit sphere bundle of $F$ if and only if
\begin{equation*}
\langle\omega^\nu_\mu,\rr\partial_\rr+\im J\rr\partial_\rr\rangle - \overline{\langle\omega^\mu_\nu,\rr\partial_\rr+\im J\rr\partial_\rr\rangle} = 0
\end{equation*}
for all $\mu, \nu$. Equivalently, in terms of the isomorphism \eqref{DualbdyIso},
\begin{equation}\label{ExactMetricCondition}
\langle(\Phi^*)^{-1}\omega^\nu_\mu,\T \rangle + \overline{\langle(\Phi^*)^{-1}\omega^\mu_\nu,\T\rangle} = 0\quad \text{ for all }\mu,\nu.
\end{equation}
\begin{definition}\label{ExactMetric}
The Hermitian metric $h$ will be called exact if \eqref{ExactMetricCondition} holds.
\end{definition}

The terminology in Definition \ref{ExactMetric} is taken from the notion of exact Riemannian $b$-metric of Melrose~\cite[pg. 31]{RBM2}. For such metrics, the Levi-Civita $b$-connection has the property that $\bnabla_{\rr\partial_\rr}=0$ [op. cit., pg. 58]. We proceed to show that the Hermitian holomorphic connection of an exact Hermitian metric on $F$ also has this property. Namely, suppose that $h$ is an exact Hermitian metric, and let $\eta_\mu$ be an orthonormal frame of $F$. Then for the Hermitian holomorphic connection we have
\begin{equation*}
\langle\omega^\nu_\mu-\overline \omega^\mu_\nu,\rr\partial_\rr\rangle = \langle\omega^\nu_\mu,\rr\partial_\rr\rangle - \overline{\langle \omega^\mu_\nu,\rr\partial_\rr\rangle}
=\frac{1}{2}\big( \langle\omega^\nu_\mu,\rr\partial_\rr+\im J\rr\partial_\rr\rangle - \overline{\langle \omega^\mu_\nu,\rr\partial_\rr+\im J\rr\partial_\rr\rangle} \big)
\end{equation*}
using that the $\omega^\nu_\mu$ are of type $(0,1)$. Thus  $\bnabla_{\rr\partial_\rr}=0$.

\section{Local invariants}\label{sLocalInvariants}

Complex structures have no local invariants: every point of a complex $n$-manifold has a neighborhood biholomorphic to a ball in $\C^n$ It is natural to ask the same question about complex $b$-structures, namely,
\begin{equation*}
\display{300pt}{is there a local model depending only on dimension for every complex $b$-stucture? }
\end{equation*}
In lieu of a Newlander-Nirenberg theorem, we show that complex $b$-structures have no local formal invariants at the boundary. More precisely:

\begin{proposition}\label{NoLocalFormalInvariants}
Every $p_0\in \N$ has a neighborhood $V$ in $\M$ on which there are smooth coordinates $x^j$, $j=1,\dotsc,2n+2$ centered at $p_0$ with $x^{n+1}$ vanishing on $V\cap\N$ such that with
\begin{equation}\label{NoLocalFormalInvariants1}
\overline L^0_j=\frac{1}{2}(\partial_{x^j}+\im\partial_{x^{j+n+1}}),\ j\leq n,\quad
\overline L^0_{n+1} = \frac{1}{2}(x^{n+1}\partial_{x^{n+1}}+\im\partial_{x^{2n+2}})
\end{equation}
there are smooth functions $\gamma^j_k$ vanishing to infinite order on $V\cap\N$ such that
\begin{equation*}
\overline L_j=\overline L_j^0+\sum_{k=1}^{n+1}\gamma^k_j L_k^0
\end{equation*}
is a frame for $\bT^{0,1}\M$ over $V$.
\end{proposition}

The proof will require some preparation. Let $\rr:\M\to\R$ be a defining function for $\partial\M$. Let $p_0\in \N$, pick a hypoanalytic chart $(z,t)$ (cf. \eqref{HypoanalyticChart}) centered at $p_0$ with $\T t=1$. Let $U\subset \N$ be a neighborhood of $p_0$ contained in the domain of the chart, mapped by it to $B\times(-\delta,\delta)\subset \C^n\times\R$, where $B$ is a ball with center $0$ and $\delta$ is some small positive number. For reference purposes we state

\begin{lemma}\label{LocalExactBStructure} On such $U$, the problem
\begin{equation*}
\Deebar \phi = \psi, \quad \psi\in C^\infty(U;\Wedge^{q+1}\smash[t]{\Veebar}^*|_U)\text{ and }\Deebar \psi=0
\end{equation*}
has a solution in $C^\infty(U;\Wedge^q\smash[t]{\Veebar}^*|_U)$.
\end{lemma}

Extend the functions $z^j$ and $t$ to a neighborhood of $p_0$ in $\M$. Shrinking $U$ if necessary, we may assume that in some neighborhood $V$ of $p_0$ in $\M$ with $V\cap\partial\M=U$, $(z,t,\rr)$ maps $V$ diffeomorphically onto $B\times(-\delta,\delta)\times [\neutral{]}0,\e\neutral{(})$ for some $\delta$, $\e>0$. Since the form $\beta_\rr$ defined in \eqref{DefinitionBeta} is $\Deebar$-closed, there is $\alpha\in C^\infty(U)$ such that
\begin{equation*}
-\im\Deebar \alpha=\beta_\rr.
\end{equation*}
Extend $\alpha$ to $V$ as a smooth function. The section
\begin{equation}\label{StartNoInv}
\bdeebar(\log \rr+\im\alpha)=\frac{\bdeebar \rr}{\rr}+\im\bdeebar\alpha
\end{equation}
of $\bWedge^{0,1}\M$ over $V$ vanishes on $U$, since $\beta_\rr+\im \Deebar \alpha=0$. So there is a smooth section $\phi$ of $\bWedge^{0,1}\M$ over $V$ such that
\begin{equation*}
\bdeebar(\log \rr +\im\alpha)=\rr e^{\im\alpha}\phi.
\end{equation*}
Suppose $\zeta:U\to\C$ is a solution of $\Deebar \zeta=0$ on $U$, and extend it to $V$. Then $\bdeebar\zeta$ vanishes on $U$, so again we have
\begin{equation*}
\bdeebar\zeta=\rr e^{\im\alpha}\psi.
\end{equation*}
for some smooth section $\psi$ of $\bWedge^{0,1}\M$ over $V$. The following lemma will be applied for $f_0$ equal to $\log\rr+\im\alpha$ or each of the functions $z^j$.

\begin{lemma}
Let $f_0$ be smooth in $V\minus U$ and suppose that $\bdeebar f_0=\rr e^{\im\alpha}\psi_1$ with $\psi_1$ smooth on $V$. Then there is $f:V\to\C$ smooth vanishing at $U$ such that $\bdeebar(f_0+f)$ vanishes to infinite order on $U$.
\end{lemma}

\begin{proof} Suppose that $f_1,\dotsc,f_{N-1}$ are defined on $V$ and that
\begin{equation}\label{StepN}
\bdeebar \sum_{k=0}^{N-1}(\rr e^{\im\alpha})^k f_k = (\rr e^{\im\alpha})^N\psi_N
\end{equation}
holds with $\psi_N$ smooth in $V$; by the hypothesis, \eqref{StepN} holds when $N=1$. Using \eqref{StartNoInv} we get that $\bdeebar(\rr e^{\im\alpha})=(\rr e^{\im\alpha})^2\phi$, therefore
\begin{equation*}
0=\bdeebar\big((\rr e^{\im\alpha})^N\psi_N) = (\rr e^{\im\alpha})^N[\bdeebar\psi_N + N\rr e^{\im\alpha}\phi\wedge \psi_N],
\end{equation*}
which implies that $\bdeebar\psi_N=0$ on $U$.  With arbitrary $f_N$ we have
\begin{equation*}
\bdeebar\sum_{k=0}^{N}(\rr e^{\im\alpha})^k f_k = (\rr e^{\im\alpha})^N(\psi_N  + \bdeebar f_N + N \rr e^{\im\alpha}f_N \phi).
\end{equation*}
Since $\Deebar\psi_N=0$ and $H_{\Deebar}^1(U)=0$ by Lemma~\ref{LocalExactBStructure}, there is a smooth function $f_N$ defined in $U$ such that $\Deebar f_N=-\psi_N$ in $U$. So there is $\chi_N$ such that $\psi_N + \bdeebar f_N = \rr e^{\im\alpha}\chi_N$. With such $f_N$, \eqref{StepN} holds with $N+1$ in place of $N$ and some $\psi_{N+1}$. Thus there is a sequence $\{f_j\}_{j=1}^\infty$ such that \eqref{StepN} holds for each $N$. Borel's lemma then gives $f$ smooth with
\begin{equation*}
f\sim \sum_{k=1}^\infty (\rr e^{\im\alpha})^k f_k\quad\text{ on }U
\end{equation*}
such that $\Deebar(f_0+f)$ vanishes to infinite order on $U$.
\end{proof}

\begin{proof}[Proof of Proposition~\ref{NoLocalFormalInvariants}]
Apply the lemma with $f_0=\log\rr+\im\alpha$ to get a function $f$ such that $\bdeebar (f_0+f)$ vanishes to infinite order at $U$. Let
\begin{equation*}
x^{n+1}=\rr e^{-\Im\alpha+\Re f},\quad x^{2n+2}=\Re\alpha+\Im f.
\end{equation*}
These functions are smooth up to $U$.

Applying the lemma to each of the functions $f_0=z^j$, $j=1,\dotsc,n$ gives smooth functions $\zeta^j$ such that $\zeta^j=z^j$ on $U$ and $\bdeebar \zeta^j=0$ to infinite order at $U$. Define
\begin{equation*}
x^j=\Re \zeta^j,\quad x^{j+n+1}=\Im\zeta^j,\quad j=1,\dotsc,n.
\end{equation*}
The functions $x^j$, $j=1\dots,2n+2$ are independent, and the forms
\begin{equation*}
\eta^j=\bd\zeta^j,j=1\dots,n,\quad \eta^{n+1}=\frac{1}{x^{n+1}e^{\im x^{2n+2}}}\bd[x^{n+1}e^{\im x^{2n+2}}]
\end{equation*}
together with their conjugates form a frame for $\C\bT\M$ near $p_0$. Let $\eta^j_{1,0}$ and $\eta^j_{0,1}$ be the $(1,0)$ and $(0,1)$ components of $\eta^j$ according to the complex $b$-structure of $\M$. Then
\begin{equation*}
\eta^j_{0,1}=\sum_k p^j_k \eta^k +q^j_k \overline {\eta}^k.
\end{equation*}
Since $\eta^j_{0,1}=\bdeebar\zeta^j$ vanishes to infinite order at $U$,
the coefficients $p^j_k$ and $q^j_k$ vanish to infinite order at $U$. Replacing this formula for $\eta^j_{0,1}$ in $\eta^j=\eta^j_{1,0}+\eta^j_{0,1}$
get
\begin{equation*}
\sum_k(\delta^j_k-p^j_k )\eta^k -\sum_k q^j_k \overline \eta^k= \eta^j_{1,0}.
\end{equation*}
The matrix $\Id-[p^j_k]$ is invertible with inverse of the form $\Id+[P^j_k]$ with $P^j_k$ vanishing to infinite order at $U$. So
\begin{equation}\label{TheForms}
\eta^j -\sum_k \gamma^j_k \overline \eta^k=\sum_k(\delta^j_k+P^j_k) \eta^k_{1,0}
\end{equation}
with suitable $\gamma^j_k$ vanishing to infinite order on $U$. Define the vector fields $\overline L_j^0$ as in \eqref{NoLocalFormalInvariants1}. The vector fields
\begin{equation*}
\overline L_j=\overline L_j^0+\sum_k \gamma^k_j L_k^0, \quad j=1,\dotsc,n+1
\end{equation*}
are independent and since $\langle \overline L_j^0,\eta^k\rangle=0$ and $\langle L_j^0,\eta^k\rangle=\delta^k_j$, they annihilate each of the forms on the left hand side of \eqref{TheForms}. So they annihilate the forms $\eta^k_{1,0}$, which proves that the $\overline L_j$ form a frame of $\bT^{0,1}\M$.
\end{proof}

\section{Indicial complexes}\label{sIndicialComplex}

Throughout this section we assume that $\N$ is a connected component of the boundary of a compact manifold $\M$. Let
\begin{equation}\label{GenericComplex}
\cdots\to C^\infty(\M;E^q)\xrightarrow{A_q} C^\infty(\M;E^{q+1})\to\cdots
\end{equation}
be a $b$-elliptic complex of operators $A_q\in \Diff^1_b(\M;E^q,E^{q+1})$; the $E^q$, $q=0,\dotsc,r$, are vector bundles over $\M$.

Note that since $A_q$ is a first order operator,
\begin{equation}\label{FirstOrderChar}
A_q(f\phi)=f A_q\phi-\im\, \bsym(A_q)(\bd f)(\phi).
\end{equation}
This formula follows from the analogous formula for the standard principal symbol and the definition of principal $b$-symbol. It follows from \eqref{FirstOrderChar} and \eqref{VanishingOnBdy} that $A_q$ defines an operator
\begin{equation*}
A_{b,q}:\Diff^1(\N;E^q_{\N},E^{q+1}_{\N}).
\end{equation*}
Fix a smooth defining function $\rr:\M\to\R$ for $\partial \M$, $\rr>0$ in the interior of $\M$, let 
\begin{equation*}
\A_q(\sigma): \Diff^1_b(\N;E^q_{\N},E^{q+1}_{\N}),\quad \sigma\in \C
\end{equation*}
denote the indicial family of $A_q$ with respect to $\rr$, see \eqref{IndicialFamily}. Using \eqref{FirstOrderChar} and defining
\begin{equation*}
\Lambda_{\rr,q}=\bsym(A_q)(\frac{\bd \rr}{\rr}),
\end{equation*}
the indicial family of $A_q$ with respect to $\rr$ is
\begin{equation}\label{DefGenericIndOp}
\A_q(\sigma) = A_{b,q}+\sigma\Lambda_{\rr,q}:C^\infty(\N;E^q_\N)\to C^\infty(\N;E^{q+1}_\N).
\end{equation}
Because of \eqref{PQb}, these operators form an elliptic complex
\begin{equation}\label{IndicialComplex}
\cdots
\to C^\infty(\N;E^q_\N) \xrightarrow{\A_q(\sigma)}
C^\infty(\N;E^{q+1}_\N)
\to\cdots
\end{equation}
for each $\sigma$ and each connected component $\N$ of $\partial \M$. The operators depend on $\rr$, but the cohomology groups at a given $\sigma$ for different defining functions $\rr$ are isomorphic. Indeed, if $\rr'$ is another defining function for $\partial \M$, then $\rr'=e^u\rr$ for some smooth real-valued function $u$, and a simple calculation gives
\begin{equation*}
(A_{b,q}+\sigma \Lambda_{\rr,q})(e^{\im\sigma u}\phi)=e^{\im\sigma u}(A_{b,q}+\sigma \Lambda_{\rr',q})\phi.
\end{equation*}
In analogy with the definition of boundary spectrum of an elliptic operator $A\in \Diff^m_b(\M;E,F)$, we have

\begin{definition}
Let $\N$ be a connected component of $\partial \M$. The family of complexes \eqref{IndicialComplex}, $\sigma\in \C$, is the indicial complex of \eqref{GenericComplex} at $\N$. For each $\sigma\in \C$ let $H^q_{\A(\sigma)}(\N)$ denote the $q$-th cohomology group of \eqref{IndicialComplex} on $\N$. The $q$-th boundary spectrum of the complex \eqref{GenericComplex} at $\N$ is the set
\begin{equation*}
\spec_{b,\N}^q(A)=\set{\sigma\in \C: H^q_{\A(\sigma)}(\N)\ne 0}.
\end{equation*}
The $q$-th boundary spectrum of $A$ is $\spec_b^q(A) = \bigcup_{\N}\spec_{b,\N}^q(A)$.
\end{definition}

The spaces $H^q_{\A(\sigma)}(\N)$ are finite-dimensional because \eqref{IndicialComplex} is an elliptic complex and $\N$ is compact. It is convenient to isolate the behavior of the indicial complex according to the components of the boundary, since the sets $\spec_{b,\N}^q(A)$ can vary drastically from component to component.

\medskip
Suppose that $\M$ is a complex $b$-manifold. Recall that since
\begin{equation*}
\bdeebar \in \Diff^1_b(\M;\bWedge^{0,q}\M,\bWedge^{0,q+1}\M),
\end{equation*}
there are induced boundary operators
\begin{equation*}
\bdeebar_b\in \Diff^1(\N;\bWedge^{0,q}_\N\M,\bWedge^{0,q+1}_\N\M)
\end{equation*}
which via the isomorphism \eqref{bdyIso} become the operators of the $\Deebar$-complex \eqref{bdyComplex}. Combining \eqref{bdeebarOnFunctions} and \eqref{bsymbdeebar} we get
\begin{equation*}
\bsym(\bdeebar)(\frac{\bd \rr}{\rr})(\phi) = \im \frac{\bdeebar\rr}{\rr}\wedge\phi
\end{equation*}
and using \eqref{DefinitionBeta} we may identify $\widehat{\bdeebar_b}(\sigma)$, given by \eqref{DefGenericIndOp}, with the operator
\begin{equation}\label{DefIndOp}
\Dbar(\sigma)\phi = \Deebar \phi + \im \sigma\beta_\rr\wedge \phi.
\end{equation}
If $E\to\M$ is a holomorphic vector bundle, then the indicial family of
\begin{equation*}
\bdeebar\in \Diff^1_b(\M;\bWedge^{0,q}\M\otimes E,\bWedge^{0,q+1}\M\otimes E)
\end{equation*}
is again given by \eqref{DefIndOp}, but using the operator $\Deebar$ of the complex \eqref{defDeebarBundle}.

\medskip
Returning to the general complex \eqref{GenericComplex}, fix a smooth positive $b$-density $\m$ on $\M$ and a Hermitian metric on each $E^q$. Let $\A_q^\star(\sigma)$ be the indicial operator of the formal adjoint, $A^\star_q$, of $A_q$. The Laplacian $\Laplacian_q$ of the complex \eqref{GenericComplex} in degree $q$ belongs to $\Diff^2_b(\M;E^q\M)$, is $b$-elliptic, and its indicial operator is
\begin{equation*}
\widehat \Laplacian_q(\sigma)=
\A_q^\star(\sigma) \A_q(\sigma)+\A_{q-1}(\sigma) \A_{q-1}^\star(\sigma).
\end{equation*}
The $b$-spectrum of $\Laplacian_q$ at $\N$, see Melrose~\cite{RBM2}, is the set
\begin{equation*}
\spec_{b,\N}(\Laplacian_q) = \set{\sigma\in \C:\widehat\Laplacian_q(\sigma):C^\infty(\N;E^q_\N)\to C^\infty(\N;E^q_\N)\text{ is not invertible}}.
\end{equation*}
Note that unless $\sigma$ is real, $\widehat \Laplacian_q(\sigma)$ is not the Laplacian of the complex \eqref{IndicialComplex}.

\begin{proposition}\label{DiscretenesOfBSpectrum}
For each $q$, $\spec_{b,\N}^q(A)\subset \spec_{b,\N}(\Laplacian_q)$. \end{proposition}

Note that the set $\spec_{b,\N}(\Laplacian_q)$ depends on the choice of Hermitian metrics and $b$-density used to construct the Laplacian, but that the subset $\spec_{b,\N}^q(A)$ is independent of such choices. For a general $b$-elliptic complex \eqref{GenericComplex} it may occur that $\spec_{b,\N}^q(A)\ne \spec_{b,\N}(\Laplacian_q)$. In Example \ref{IndComplexbd} we show that $\spec_{b,\N}^q(\bd)\subset \set{0}$. As is well known, $\spec_{b,\N}(\Delta_q)$ is an infinite set if $\dim\M>1$. At the end of this section we will give an example where $\spec_{b,\N}^0(\bdeebar)$ is an infinite set. A full discussion of $\spec_{b,\N}^q(\bdeebar)$ for any $q$ and other aspects of the indicial complex of complex $b$-structures is given in Section~\ref{sIndicialCohomology}.

\begin{proof}[Proof of Proposition~\ref{DiscretenesOfBSpectrum}]
Since $\Laplacian_q$ is $b$-elliptic, the set $\spec_{b,\N}(\Laplacian_q)$ is closed and discrete. Let $H^2(\N;E^q_\N)$ be the $L^2$-based Sobolev space of order $2$. For $\sigma\notin \spec_{b,\N}(\Laplacian_q)$ let
\begin{equation*}
\G_q(\sigma):L^2(\N;E^q_\N)\to H^2(\N;E^q_\N)
\end{equation*}
be the inverse of $\widehat \Laplacian_q(\sigma)$. The map $\sigma\mapsto \G_q(\sigma)$ is meromorphic with poles in $\spec_b(\Laplacian_q)$. Since
\begin{equation*}
\A_q^\star(\sigma)= [\A_q(\overline \sigma)]^\star
\end{equation*}
the operators $\widehat \Laplacian_q(\sigma)$ are the Laplacians of the complex \eqref{IndicialComplex} when $\sigma$ is real. Thus for $\sigma\in \R\minus (\spec_{b,\N}(\Laplacian_q)\cup\spec_{b,\N}(\Laplacian_{q+1}))$ we have
\begin{equation*}
\A_{q}(\sigma) \G_{q}(\sigma)=\G_{q+1}(\sigma)\A_q(\sigma),\quad
\A_q(\sigma)^\star \G_{q+1}(\sigma) = \G_q(\sigma)\A_q^\star(\sigma)
\end{equation*}
by standard Hodge theory. Since all operators depend holomorphically on $\sigma$, the same equalities hold for $\sigma\in \mathfrak R=\C\minus(\spec_{b,\N}(\Laplacian_q)\cup \spec_{b,\N}(\Laplacian_{q+1}))$. It follows that
\begin{equation*}
\A_q^\star(\sigma)\A_q(\sigma) \G_q(\sigma) =
\G_q(\sigma) \A_q^\star(\sigma)\A_q(\sigma)
\end{equation*}
in $\mathfrak R$. By analytic continuation the equality holds on all of $\C\minus \spec_{b,\N}(\Laplacian_q)$. Thus if $\sigma_0\notin \spec_{b,\N}(\Laplacian_q)$ and $\phi$ is a $\A_q(\sigma_0)$-closed section, $\A_q(\sigma_0)\phi=0$, then the formula
\begin{equation*}
\phi=[\A_q^\star(\sigma_0) \A_q(\sigma_0)+\A_{q-1}(\sigma_0) \A_{q-1}^\star(\sigma_0)]\G_q(\sigma_0) \phi
\end{equation*}
leads to
\begin{equation*}
\phi=\A_{q-1}(\sigma_0)[\A_{q-1}^\star(\sigma_0) \G_{q}(\sigma_0)\phi].
\end{equation*}
Therefore $\sigma_0\notin \spec_{b,\N}^q(A)$.
\end{proof}

Since $\Laplacian_q$ is $b$-elliptic, the set $\spec _{b,\N}(\Laplacian_q)$ is discrete and intersects each horizontal strip $a\leq\Im\sigma\leq b$ in a finite set (Melrose~\cite{RBM2}). Consequently:

\begin{corollary}
The sets $\spec_{b,\N}^q(A)$, $q=0,1\dotsc$, are closed, discrete, and intersect each horizontal strip $a\leq\Im\sigma\leq b$ in a finite set.
\end{corollary}

We note in passing that the Euler characteristic of the complex \eqref{IndicialComplex} vanishes for each $\sigma$. Indeed, let $\sigma_0\in \C$. The Euler characteristic of the $\A(\sigma_0)$-complex is the index of
\begin{equation*}
\A(\sigma_0) + \A(\sigma_0)^\star:\bigoplus_{q\text{ even}}C^\infty(\N;E^q)\to \bigoplus_{q\text{ odd}}C^\infty(\N;E^q).
\end{equation*}
The operator $\A_q(\sigma)$ is equal to $A_{b,q} + \sigma \Lambda_{\rr,q}$, see \eqref{DefGenericIndOp}. Thus $\A_q(\sigma)^\star = A_{b,q}^\star+\overline \sigma \Lambda_{\rr,q}^\star$, and it follows that for any $\sigma$,
\begin{equation*}
\A(\sigma) + \A(\sigma)^\star=
\A(\sigma_0) + \A(\sigma_0)^\star
+(\sigma-\sigma_0)\Lambda_\rr+(\overline \sigma-\overline \sigma_0)\Lambda_\rr^\star
\end{equation*}
is a compact perturbation of $\A(\sigma_0) + \A(\sigma_0)^\star$. Therefore, since the index is invariant under compact perturbations, the index of $\A(\sigma) + \A(\sigma)^\star$ is independent of $\sigma$. Then it vanishes, since it vanishes when $\sigma\notin \bigcup_q \spec_{b,\N}^q(A)$.

\medskip
Let $\Mero^q(\N)$ be the sheaf of germs of $C^\infty(\N;E^q)$-valued meromorphic functions on $\C$ and let $\Hol^q(\N)$ be the subsheaf of germs of holomorphic functions. Let $\Sing^q(\N)=\Mero^q(\N)/\Hol^q(\N)$. The holomorphic family $\sigma\mapsto \A_q(\sigma)$ gives a sheaf homomorphism $\A_q:\Mero^q(\N)\to\Mero^{q+1}(\N)$ such that $\A_q(\Hol^q(\N))\subset \Hol^{q+1}(\N)$ and $\A_{q+1}\circ \A_{q}=0$, so we have a complex
\begin{equation}\label{BdyMeroCohomologySheaf}
\cdots\to \Sing^q(\N)\xrightarrow{\A_q}\Sing^{q+1}(\N)\to\cdots.
\end{equation}

The cohomology sheafs $\mathfrak H^q_A(\N)$ of this complex contain more refined information about the cohomology of the complex $A$.

\begin{proposition}
The sheaf $\mathfrak H^q_\A(\N)$ is supported on $\spec_{b,\N}^q(A)$.
\end{proposition}

\begin{proof}
Let $\sigma_0\in \C$ be such that $H^q_{\A(\sigma_0)}(\N)=0$ and let
\begin{equation}\label{SingPart}
\phi(\sigma)=\sum_{k=1}^{\mu}\frac{\phi_k}{(\sigma-\sigma_0)^{k}},
\end{equation}
$\mu>0$, $\phi_k\in C^\infty(\N;\Wedge^q\smash[t]{\Veebar}^*)$, represent the $\A$-closed element $[\phi]$ of the stalk of $\Sing^q(\N)$ over $\sigma_0$. The condition that $\A_q[\phi]=0$ means that $\A_q(\sigma)\phi(\sigma)$ is holomorphic, that is,
\begin{equation*}
\frac{\A_q(\sigma_0)\phi_\mu}{(\sigma-\sigma_0)^\mu} +\sum_{k=1}^{\mu-1}\frac{\A_q(\sigma_0)\phi_k+\Lambda_{\rr,q}\phi_{k+1}}{(\sigma-\sigma_0)^{k}}=0.
\end{equation*}
In particular $\A_q(\sigma_0)\phi_\mu=0$. Since $H^q_{\A(\sigma_0)}(\N)=0$, there is $\psi_\mu\in C^\infty(\N;E^{q-1})$ such that $\A_{q-1}(\sigma_0)\psi_\mu=\phi_\mu$. This shows that if $\mu=1$, then $[\phi]$ is exact, and that if $\mu>1$, then letting $\phi'(\sigma)=\phi(\sigma)-\A_{q-1}(\sigma)\psi_\mu/(\sigma-\sigma_0)^\mu$, that $\phi$ is cohomologous to an element $[\phi']$ represented by a sum as in \eqref{SingPart} with $\mu-1$ instead of $\mu$. By induction, $[\phi]$ is exact.
\end{proof}

\begin{definition}\label{CohomologySheafs}
The cohomology sheafs $\mathfrak H^q_{A}(\N)$ of the complex \eqref{BdyMeroCohomologySheaf} will be referred to as the indicial cohomology sheafs of the complex $A$. If $[\phi]\in \mathfrak h^q_A(\N)$ is a nonzero element of the stalk over $\sigma_0$, the smallest $\mu$ such that there is a meromorphic function \eqref{SingPart} representing $[\phi]$ will be called the order of the pole of $[\phi]$.
\end{definition}

The relevancy of this notion of pole lies in that it predicts, for any given cohomology class of the complex $A$, the existence of a representative with the most regular leading term (the smallest power of log that must appear in the expansion at the boundary). We will see later (Proposition~\ref{bdeebarCohomologySheaf}) that for the $b$-Dolbeault complex, under a certain geometric assumption, the order of the pole of $[\phi]\in \mathfrak H^q_{\bdeebar}(\N)\minus 0$ is $1$.

\begin{example}\label{IndComplexbd}
For the $b$-de Rham complex one has $\spec_{b,\N}^q(\bd)\subset \set{0}$ and
\begin{equation*}
H^q_{\D(0)}(\N)=H^q_{\dR}(\N)\oplus H^{q-1}_{\dR}(\N)
\end{equation*}
for each component $\N$ of $\partial\M$, and that every element of the stalk of $\mathfrak H^q_{\bd}(\N)$ over $0$ has a representative with a simple pole. By way of the residue we get an isomorphism from the stalk over $0$ onto $H^q_{\dR}(\N)$.

Since the map \eqref{evbM} is surjective with kernel spanned by $\rr\partial_\rr$, the dual map
\begin{equation}\label{evbMDual}
\ev_{\N}^*: T^*{\N}\to \bT_{\N}^*\M
\end{equation}
is injective with image the annihilator, $\Hor$, of $\rr\partial\rr$. Let $\inner_{\rr\partial_\rr}:\bWedge_{\N}^q\M\to \bWedge_{\N}^{q-1}\M$ denote interior multiplication by $\rr\partial_\rr$ Then $\Wedge^q\Hor=\ker(\inner_{\rr\partial_\rr}:\bWedge_{\N}^q\M\to \bWedge_{\N}^{q-1}\M)$. The isomorphism \eqref{evbMDual} gives isomorphisms
\begin{equation*}
\ev_{\N}^*: \Wedge^q{\N}\to \Hor^q
\end{equation*}
for each $q$. Fix a defining function $\rr$ for $\N$ and let $\Pi:\bWedge_{\N}^q\M \to \bWedge_{\N}^q\M$ be the projection on $\Hor^q$ according to the decomposition
\begin{equation*}
\bWedge_{\N}^q\M=\Hor^q\oplus \frac{\bd\rr}{\rr}\wedge\Hor^{q-1},
\end{equation*}
that is,
\begin{equation*}
\Pi\phi = \phi-\frac{\bd\rr}{\rr}\wedge \inner_{\rr\partial_\rr}\phi.
\end{equation*}
If $\phi^0\in C^\infty(\N,\Hor^q)$ and $\phi^1\in C^\infty(\N,\Hor^{q-1})$, then
\begin{equation*}
\bd(\phi^0+\frac{\bd\rr}{\rr}\wedge\phi^1)=\Pi\,\bd\phi^0+\frac{\bd\rr}{\rr}\wedge(-\Pi\,\bd\phi^1).
\end{equation*}
Since
\begin{equation*}
\rr^{-\im\sigma}\bd\rr^{\im\sigma}\phi = \bd \phi +\im \sigma\frac{\bd\rr}{\rr}\wedge\phi,
\end{equation*}
the indicial operator $\D(\sigma)$ of $\bd$ is
\begin{equation*}
\D(\sigma)(\phi_0+\frac{\bd\rr}{\rr}\wedge\phi^1) = \Pi\,\bd\phi^0+\frac{\bd\rr}{\rr}\wedge(\im \sigma\phi^0-\Pi\, \bd\phi^1).
\end{equation*}
If $\D(\sigma)(\phi_0+\frac{\bd\rr}{\rr}\wedge\phi^1)=0$, then of course $\Pi\bd\phi^0=0$ and $\im \sigma\phi^0=\Pi\bd\phi^1$, and it follows that if $\sigma\ne 0$, then
\begin{equation*}
(\phi_0+\frac{\bd\rr}{\rr}\wedge\phi^1) = \D(\sigma)\frac{1}{\im\sigma}\phi^1.
\end{equation*}
Thus all cohomology groups of the complex $\D(\sigma)$ vanish if $\sigma\ne 0$, \ie, $\spec_{b,\N}^q(\bd)\subset \set{0}$.

It is not hard to verify that
\begin{equation*}
\Pi\bd\,\ev_{\N}^*=\ev_{\N}^*d.
\end{equation*}
Since
\begin{equation*}
\rr^{-\im\sigma}\bd\rr^{\im\sigma}\phi = \bd \phi +\im \sigma\frac{\bd\rr}{\rr}\wedge\phi,
\end{equation*}
the indicial operator of $\bd$ at $\sigma=0$ can be viewed as the operator
\begin{equation*}
\begin{bmatrix}
d & 0\\
0 & -d
\end{bmatrix}:
\begin{matrix}
\Wedge^q\N\\
\oplus\\
\Wedge^{q-1}\N
\end{matrix} \to
\begin{matrix}
\Wedge^q\N\\
\oplus\\
\Wedge^{q-1}\N
\end{matrix}.
\end{equation*}
From this we get the cohomology groups of $\D(0)$ in terms of the de Rham cohomology of $\N$:
\begin{equation*}
H^q_{\D(0)}(\N)=H^q_{\dR}(\N)\oplus H^{q-1}_{\dR}(\N).
\end{equation*}
Thus the groups $H^q_{\D(0)}(\N)$ do not vanish for $q=0$, $1$, $\dim\M-1$, $\dim\M$ but may vanish for other values of $q$.

We now show that every element of the stalk of $\mathfrak H^q_{\bd}(\N)$ over $0$ has a representative with a simple pole at $0$. Suppose that
\begin{equation}\label{Representative}
\phi(\sigma)=\sum_{k=1}^\mu \frac{1}{\sigma^k}\left(\phi^0_k+\frac{\bd \rr}{\rr}\wedge \phi^1_k\right)
\end{equation}
is such that $\D(\sigma)\phi(\sigma)$ is holomorphic. Then
\begin{equation*}
\sum_{k=1}^\mu \frac{1}{\sigma^k}\left(d\phi^0_k - \frac{\bd \rr}{\rr}\wedge d\phi^1_k\right) + \frac{\bd \rr}{\rr}\wedge\left(\sum_{k=1}^{\mu-1} \frac{\im}{\sigma^k}\phi^0_{k+1}\right) = 0,
\end{equation*}
hence $d\phi^0_1=0$, $d\phi^1_\mu=0$ and $\phi^0_k=-\im d\phi^1_{k-1}$, $k=2,\dotsc,\mu$. Let
\begin{equation*}
\psi(\sigma)=-\im \sum_{k=2}^{\mu+1} \frac{1}{\sigma^k} \phi^1_{k-1}.
\end{equation*}
Then
\begin{align*}
\D(\sigma)\psi(\sigma)
&= -\im \sum_{k=2}^{\mu+1} \frac{1}{\sigma^k} d\phi^1_{k-1} + \frac{\bd\rr}{\rr}\wedge\sum_{k=2}^{\mu+1}\frac{1}{\sigma^{k-1}}\phi^1_{k-1}\\
&= \sum_{k=2}^{\mu} \frac{1}{\sigma^k} \phi^0_k + \frac{\bd\rr}{\rr}\wedge\sum_{k=1}^{\mu}\frac{1}{\sigma^k}\phi^1_k
\end{align*}
so
\begin{equation*}
\phi(\sigma)-\D(\sigma)\psi(\sigma)=\frac{1}{\sigma}\phi^0_1.
\end{equation*}
The map that sends the class of the $\D(\sigma)$-closed element \eqref{Representative} to the class of $\phi^0_1$ in $H^q_{\dR}(\N)$ is an isomorphism.
\end{example}

\begin{example}
As we just saw, the boundary spectrum of the $\bd$ complex in degree $0$ is just $\set{0}$. In contrast, $\spec_{b,\N}^0(\bdeebar)$ may be an infinite set. We illustrate this in the context of Example \ref{AnisotropicSphere}. The functions
\begin{equation*}
z^\alpha=(z^1)^{\alpha_1}\dotsm (z^{n+1})^{\alpha_{n+1}},
\end{equation*}
where the $\alpha_j$ are nonnegative integers, are CR functions that satisfy
\begin{equation*}
\T z^\alpha=\im (\sum \tau_j\alpha_j) z^\alpha.
\end{equation*}
This implies that
\begin{equation*}
\Deebar z^\alpha + \im (-\im \sum \tau_j\alpha_j) \beta z^\alpha=0
\end{equation*}
with $\beta$ as in Example \ref{AnisotropicSphere}, so the numbers $\sigma_\alpha=(-\im \sum \tau_j\alpha_j)$ belong to $\spec_{b,\N}^0(\bdeebar)$. 

For the sake of completeness we also show that if $\sigma\in \spec_{b,\N}^0(\bdeebar)$, then $\sigma=\sigma_\alpha$ for some $\alpha$ as above. To see this, suppose that $\zeta:S^{2n+1}\to\C$ is not identically zero and satisfies
\begin{equation*}
\Deebar \zeta+\im \sigma\zeta\beta=0
\end{equation*}
for some $\sigma\ne 0$. Then $\zeta$ is smooth, because the principal symbol of $\Deebar$ on functions is injective. Since $\langle\beta,\T\rangle=-\im$,
\begin{equation*}
T\zeta+ \sigma\zeta=0.
\end{equation*}
Thus $\zeta(\a_t(p))=e^{-\sigma t}\zeta(p)$ for any $p$. Since $|\zeta(\a_t(p))|$ is bounded as a function of $t$ and $\zeta$ is not identically $0$, $\sigma$ must be purely imaginary. Since $\zeta$ is a CR function, it extends uniquely to a holomorphic function $\tilde\zeta$ on $B=\set{z\in \C^{n+1}:\|z\|<1}$, necessarily smooth up to the boundary. Let $\zeta_t=\zeta\circ \a_t$. This is also a smooth CR function, so it has a unique holomorphic extension $\tilde \zeta_t$ to $B$. The integral curve through $z_0=(z^1_0,\dotsc,z^{n+1}_0)$ of the vector field $\T$ is
\begin{equation*}
t\mapsto \a_t(z_0)=(e^{\im \tau_1 t} z^1_0,\dotsc,e^{\im \tau_{n+1} t} z^{n+1}_0)
\end{equation*}
Extending the definition of $\a_t$ to allow arbitrary $z\in \C^{n+1}$ as argument we then have that $\tilde \zeta_t=\tilde \zeta\circ \a_t$. Then
\begin{equation*}
\partial_t\tilde \zeta_t+ \sigma\tilde \zeta_t=0
\end{equation*}
gives
\begin{equation*}
\tilde\zeta(z)=\sum_{\set{\alpha:\pmb \tau\cdot\alpha=\im \sigma}} c_\alpha z^\alpha
\end{equation*}
for $|z|<1$, where $\pmb\tau=(\tau_1,\dotsc,\tau_{n+1})$. Thus $\sigma=-\im\sum\tau_j\alpha_j$ as claimed. Note that $\Im \sigma$ is negative (positive) if the $\tau_j$ are positive (negative) and $\alpha\ne0$.
\end{example}

\section{Underlying CR complexes}\label{sUnderlyingCRcomplexes}

Again let $\a:\R\times\N\to\N$ be the flow of $\T$. Let $\Lie_\T$ denote the Lie derivative with respect to $\T$ on de Rham $q$-forms or vector fields and let $\inner_\T$ denote interior multiplication by $\T$ of de Rham $q$-forms or of elements of $\Wedge^q\smash[t]{\Veebar}^*$.

The proofs of the following two lemmas are elementary.

\begin{lemma}
If $\alpha$ is a smooth section of the annihilator of $\Veebar$ in $\C T^*\N$, then $(\Lie_{\T}\alpha)|_{\Veebar}=0$. Consequently, for each $p\in \N$ and $t\in \R$, $d\a_t:\C T_p\N\to \C T_{\a_t(p)}\N$ maps $\Veebar_p$ onto $\Veebar_{\a_t(p)}$.
\end{lemma}

It follows that there is a well defined smooth bundle homomorphism $\a_t^*: \Wedge^q\smash[t]{\Veebar}^* \to \Wedge^q\smash[t]{\Veebar}^*$ covering $\a_{-t}$. In particular, one can define the Lie derivative $\Lie_{\T}\phi$ with respect to $\T$ of an element in $\phi \in C^\infty(\N;\Wedge^q\smash[t]{\Veebar}^*)$. The usual formula holds:

\begin{lemma}
If $\phi\in C^\infty(\N;\Wedge^q\smash[t]{\Veebar}^*)$, then $\Lie_{\T}\phi=\inner_{\T}\Deebar\phi+\Deebar\inner_{\T}\phi$. Consequently, for each $t$ and $\phi\in C^\infty(\N;\Wedge^q\smash[t]{\Veebar}^*)$, $\Deebar \a_t^*\phi = \a_t^*\Deebar\phi$.
\end{lemma}

\medskip
For any defining function $\rr$ of $\N$ in $\M$, $\Kbar_{\rr}=\ker\beta_\rr$ is a CR structure of CR codimension $1$: indeed, $\K_\rr\cap \Kbar_\rr\subset \Span_\C\T$ but since $\langle\beta_\rr,\T\rangle$ vanishes nowhere, we must have $\Kbar\cap \K=0$. Since $\K\oplus \Kbar\oplus\Span_\C\T=\C T\N$, the CR codimension is $1$. Finally, if $V,W\in C^\infty(\N;\Kbar_{\rr})$, then 
\begin{equation*}
\langle\beta_\rr,[V,W]\rangle=V\langle\beta_\rr,W\rangle-W\langle\beta_\rr,V\rangle-2\Deebar\beta(V,W),
\end{equation*}
Since the right hand side vanishes, $[V,W]$ is again a section of $\Kbar_\rr$.

Since $\Veebar=\Kbar_\rr\oplus \Span_\C \T$, the dual of $\Kbar_\rr$ is canonically isomorphic to the kernel of $\inner_\T:\smash[t]{\Veebar}^*\to\C$. We will write $\Kbar^*$ for this kernel. More generally, $\Wedge^q\Kbar_\rr^*$ and the kernel, $\Wedge^q\Kbar^*$, of $\inner_\T:\Wedge^q\smash[t]{\Veebar}^*\to\Wedge^{q-1}\smash[t]{\Veebar}^*$ are canonically isomorphic. The vector bundles $\Wedge^q\Kbar^*$ are independent of the defining function $\rr$. We regard the $\deebarb$-operators of the CR structure as operators
\begin{equation*}
C^\infty(\N;\Wedge^q\Kbar^*)\to C^\infty(\N;\Wedge^{q+1}\Kbar^*).
\end{equation*}
They do depend on $\rr$ but we will not indicate this in the notation. 

To get a formula for $\deebarb$, let
\begin{equation*}
\tilde \beta_\rr=\frac{\im }{\im -a_\rr}\beta_\rr
\end{equation*}
(so that $\langle \im \tilde \beta_\rr,\T\rangle=1$). The projection $\Pi_\rr:\Wedge^q \smash[t]{\Veebar}^*\to\Wedge^q\smash[t]{\Veebar}^*$ on $\Wedge^q\Kbar^*$ according to the decomposition
\begin{equation}\label{DecompositionOfVeebar}
\Wedge^q\smash[t]{\Veebar}^*=\Wedge^q \Kbar^* \oplus \im \tilde \beta_\rr\wedge \Wedge^{q-1}\Kbar^*
\end{equation}
is
\begin{equation}\label{DefinitionOfPi}
\Pi_\rr\phi=\phi-\im\tilde \beta_\rr\wedge \inner_\T\phi.
\end{equation}

\begin{lemma}
With the identification of $\Wedge^q\Kbar_\rr^*$ with $\Wedge^q\Kbar^*$ described above, the $\deebarb$-operators of the CR structure $\Kbar_\rr$ are given by
\begin{equation}\label{DefinitionOfdeebarb}
\deebarb\phi=\Pi_\rr\Deebar\phi\quad \text{ if }\phi\in C^\infty(\N,\Wedge^q\Kbar^*),
\end{equation}
\end{lemma}

\begin{proof} Suppose that $(z,t)$ is a hypoanalytic chart for $\Veebar$ on some open set $U$, with $\T t=1$. So $\partial_{\overline z^\mu}$, $\mu=1\dotsc,n$, $\T=\partial_t$ is a frame for $\Veebar$ over $U$ with dual frame $\Deebar\overline z^\mu$, $\Deebar t$. If
\begin{equation*}
\beta_\rr=\sum_{\mu=1}^n\beta_\mu\Deebar \overline z^\mu+\beta_0\Deebar t.
\end{equation*}
then
\begin{equation*}
\overline L_\mu=\partial_{\overline z^\mu}-\frac{\beta_\mu}{\beta_0}\partial_t,\quad\mu=1,\dotsc,n
\end{equation*}
is a frame for $\Kbar_\rr$ over $U$. Let $\overline \eta^\mu$ denote the dual frame (for $\Kbar_\rr^*$). Since the $\overline L_\mu$ commute, $\deebarb\overline \eta^\mu=0$, so if $\phi=\sum'_{|I|=q}\phi_I\,\overline \eta^I$, then (with the notation as in eg. Folland and Kohn~\cite{FK})
\begin{equation*}
\deebarb \phi=\sideset{}{'}\sum_{|J|=q+1}\,\sideset{}{'}\sum_{|I|=q}\sum_\mu \sign {\mu I}{J}\overline L_\mu\phi_I\,\overline \eta^J.
\end{equation*}
On the other hand, the frame of $\smash[t]{\Veebar}^*$ dual to the frame $\overline L_\mu$, $\mu=1,\dotsc,n$, $\T$ of $\Veebar$ is $\Deebar \overline z^\mu$, $\im\tilde \beta_\rr$, and the identification of $\Kbar_\rr^*$ with $\Kbar^*$ maps the $\eta^\mu$ to the $\Deebar \overline z^\mu$. So, as a section of $\Wedge^q\smash[t]{\Veebar}^*$,
\begin{equation*}
\phi=\sideset{}{'}\sum_{|I|=q} \phi_I\, \Deebar\overline z^I
\end{equation*}
and
\begin{equation*}
\Deebar\phi=\sideset{}{'}\sum_{|J|=q+1}\,\sideset{}{'}\sum_{|I|=q} \sign{\mu I}{J} \overline L_\mu\phi_I \,\Deebar\overline z^J+
\im\tilde \beta_\rr\wedge\sideset{}{'}\sum_{|I|=q} \T\phi_I\,\Deebar\overline z^I.
\end{equation*}
Thus $\Pi_\rr\Deebar\phi$ is the section of $\Wedge^{q+1}\Kbar^*$ associated with $\deebarb\phi$ by the identifying map.
\end{proof}

Using \eqref{DefinitionOfPi} in \eqref{DefinitionOfdeebarb} and the fact that $\inner_\T\Deebar\phi=\Lie_\T\phi$ if $\phi\in C^\infty(\N;\Wedge^q\Kbar^*)$ we get
\begin{equation}\label{FormulaForDeeOnK*}
\deebarb\phi=\Deebar\phi-\im \tilde \beta_\rr\wedge\Lie_\T\phi \quad \text{ if }\phi\in C^\infty(\N,\Wedge^q\Kbar^*).
\end{equation}

The $\Deebar$ operators can be expressed in terms of the $\deebarb$ operators. Suppose $\phi\in C^\infty(\N;\Wedge^q\smash[t]{\Veebar}^*)$. Then $\phi=\phi^0+\im \tilde \beta_\rr\wedge \phi^1$ with unique $\phi^0 \in C^\infty(\N;\Wedge^q \Kbar^*)$ and $\phi^1 \in C^\infty(\N;\Wedge^{q-1} \Kbar^*)$, and
\begin{equation*}
\Deebar\phi^0=\deebarb\phi^0+\im\tilde \beta_\rr\wedge\Lie_\T\phi^0,
\end{equation*}
see \eqref{FormulaForDeeOnK*}. Using
\begin{equation*}
\Deebar\tilde \beta_\rr=\frac{\Deebar a_\rr}{\im - a_\rr}\wedge\tilde \beta_\rr
\end{equation*}
and \eqref{FormulaForDeeOnK*} again we get
\begin{equation*}
\Deebar(\im \tilde \beta_\rr\wedge \phi^1)=
\im\tilde \beta_\rr\wedge \big(-\frac{\Deebar a_\rr}{\im - a_\rr}\wedge\phi^1 - \Deebar\phi^1 \big)=
\im\tilde \beta_\rr\wedge \big(-\frac{\deebarb a_\rr}{\im - a_\rr}\wedge\phi^1 - \deebarb\phi^1 \big).
\end{equation*}
This gives
\begin{equation}\label{DeeAsMatrix}
\Deebar=
\begin{bmatrix}
\deebarb & 0\\
\Lie_\T & -\deebarb -\dfrac{\deebarb a_\rr}{\im -a_\rr}
\end{bmatrix}
:
\begin{matrix}
C^\infty(\N;\Wedge^q\Kbar^*)\\ \oplus \\ C^\infty(\N;\Wedge^{q-1}\Kbar^*)
\end{matrix}
\to
\begin{matrix}
C^\infty(\N;\Wedge^{q+1}\Kbar^*)\\ \oplus \\ C^\infty(\N;\Wedge^q\Kbar^*)
\end{matrix}.
\end{equation}

Since $\T$ itself is $\T$-invariant, $\inner_\T\a_t^*=a_t^*\inner_\T$: the subbundle $\Kbar^*$ of $\smash[t]{\Veebar}^*$ is invariant under $\a_t^*$ for each $t$. This need not be true of $\Kbar_\rr$, \ie, the statement that for all $t$, $d\a_t(\Kbar_\rr)\subset \Kbar_\rr$, equivalently,
\begin{equation*}
L\in C^\infty(\M;\Kbar_\rr)\implies [\T,L]\in C^\infty(\M;\Kbar_\rr),
\end{equation*}
may fail to hold. Since $\Deebar\beta_\rr=0$, the formula
\begin{equation*}
0=\T\langle \beta_\rr,L\rangle - L\langle \beta_\rr,\T\rangle - \langle \beta_\rr,[\T,L]\rangle
\end{equation*}
with $L\in C^\infty(\N;\Kbar_\rr)$ gives that $\Kbar_\rr$ is invariant under $d\a_t$ if and only if $La_\rr=0$ for each CR vector field, that is, if and only if $a_\rr$ is a CR function. This proves the equivalence between the first and last statements in the following lemma. The third statement is the most useful.

\begin{lemma}\label{Invariances}
Let $\rr$ be a defining function for $\N$ in $\M$ and let $\deebarb$ denote the operators of the associated CR complex. The following are equivalent:
\begin{enumerate}
\item The function $a_\rr$ is CR;
\item $\Lie_\T\tilde \beta_\rr=0$;
\item $\Lie_\T\deebarb-\deebarb\Lie_\T=0$;
\item $\Kbar_\rr$ is $\T$-invariant.
\end{enumerate}
\end{lemma}

\begin{proof}
From $\beta_\rr=(a_\rr-\im)\im\tilde \beta_\rr$ and $\Lie_\T\beta_\rr=\Deebar a_\rr$ we obtain
\begin{equation*}
\Deebar a_\rr=(\Lie_\T a_\rr)\im \tilde \beta_\rr+(a_\rr-\im)\im \Lie_\T\tilde \beta_\rr,
\end{equation*}
so
\begin{equation*}
\deebarb a_\rr=\Deebar a_\rr-(\Lie_\T a_\rr) \im \tilde \beta_\rr =(a_\rr-\im)\im \Lie_\T\tilde \beta_\rr.
\end{equation*}
Thus $a_\rr$ is CR if and only if $\Lie_\T\tilde \beta_\rr=0$.

Using $\Lie_\T\Deebar=\Deebar\Lie_\T$ and the definition of $\deebarb$ we get
\begin{equation*}
\Lie_\T \deebarb\phi
=\Lie_\T(\Deebar\phi-\im\tilde \beta_\rr\wedge \Lie_\T\phi)
=\deebarb\Lie_\T\phi-\im(\Lie_\T\tilde \beta_\rr)\wedge \Lie_\T\phi
\end{equation*}
for $\phi\in C^\infty(\N;\Wedge^q\Kbar^*)$. Thus $\Lie_\T\deebarb-\deebarb\Lie_\T=0$ if and only if $\Lie_\T\tilde \beta_\rr=0$.
\end{proof}

\begin{lemma}\label{WithInvariantMetric}
Suppose that $\Veebar$ admits a $\T$-invariant metric. Then there is a defining function $\rr$ for $\N$ in $\M$ such that $a_\rr$ is constant. If $\rr$ and $\rr'$ are defining functions such that $a_\rr$ and $a_{\rr'}$ are constant, then $a_\rr = a_{\rr'}$. This constant will be denoted $\a_\av$.
\end{lemma}

\begin{proof}
Let $h$ be a metric as stated. Let $\Hor^{0,1}$ be the subbundle of $\Veebar$ orthogonal to $\T$. This is $\T$-invariant, and since the metric is $\T$-invariant, $\Hor^{0,1}$ has a $\T$-invariant metric. This metric gives canonically a metric on $\Hor^{1,0}=\overline {\Hor^{0,1}}$. Using the decomposition $\C T\N=\Hor^{1,0}\oplus\Hor^{0,1}\oplus \Span_\C \T$ we get a $\T$-invariant metric on $\C T\N$ for which the decomposition is orthogonal. This metric is induced by a Riemannian metric $g$. Let $\m_0$ be the corresponding Riemannian density, which is $\T$-invariant because $g$ is. Since $\Deebar$, $h$, and $\m_0$ are $\T$-invariant, so are the formal adjoint $\Deebar^\star$ of $\Deebar$ and the Laplacians of the $\Deebar$-complex, and if $G$ denotes the Green's operators for these Laplacians, then $G$ is also $\T$-invariant, as is the orthogonal projection $\Pi$ on the space of $\Deebar$-harmonic forms. Arbitrarily pick a defining function $\rr$ for $\N$ in $\M$. Then
\begin{equation*}
a_\rr-G\Deebar^\star\Deebar a_\rr=\Pi a_\rr
\end{equation*}
where $\Pi a_\rr$ is a constant function by Lemma~\ref{ConstantSolutions}.
Since $\beta_\rr$ is $\Deebar$-closed, $\Deebar a_\rr=\Lie_\T\beta_\rr$. Thus $G\Deebar^\star\Deebar a_\rr= \T G\Deebar^\star \beta_\rr$, and since $a_\rr$ is real valued and $\T$ is a real vector field,
\begin{equation*}
a_\rr- \T \Re G\Deebar^\star\beta_\rr=\Re\Pi a_\rr.
\end{equation*}
Extend the function $u=\Re G\Deebar^\star\beta_\rr$ to $\M$ as a smooth real-valued function. Then $\rr'=e^{-u}\rr$ has the required property.

Suppose that $\rr$, $\rr'$ are defining functions for $\N$ in $\M$ such that $a_\rr$ and $a_{\rr'}$ are constant. Then these functions are equal by Proposition \ref{Averages}. 
\end{proof}

Note that if for some $\rr$, the subbundle $\Kbar_\rr$ is $\T$-invariant and admits a $\T$ invariant Hermitian metric, then there is a $\T$-invariant metric on $\Veebar$.

\medskip
Suppose now that $\rho:F\to\M$ is a holomorphic vector bundle over $\M$. Using the operators
\begin{equation*}
\Deebar:C^\infty(\N;\Wedge^q\smash[t]{\Veebar}^*\otimes F_\N)\to C^\infty(\N;\Wedge^{q+1}\smash[t]{\Veebar}^*\otimes F_\N),
\end{equation*}
see \eqref{defDeebarE}, define operators
\begin{equation}\label{deebarbE}
\cdots\to C^\infty(\N;\Wedge^q\Kbar^*\otimes F_\N)\xrightarrow{\deebarb} C^\infty(\N;\Wedge^{q+1}\Kbar^*\otimes F_\N)\to\cdots
\end{equation}
by
\begin{equation*}
\deebarb \phi =\Pi_\rr\Deebar\phi,\quad\phi\in C^\infty(\N;\Wedge^q\Kbar^*\otimes F_\N)
\end{equation*}
where $\Pi_\rr$ means $\Pi_\rr\otimes \Id$ with $\Pi_\rr$ defined by \eqref{DefinitionOfPi}. The operators \eqref{deebarbE} form a complex. Define also
\begin{equation*}
\Lie_\T = \inner_\T\Deebar + \Deebar\inner_\T
\end{equation*}
where $\inner_\T$ stands for $\inner_\T\otimes \Id$. Then
\begin{equation*}
\inner_\T \Lie_\T = \Lie_\T\inner_\T,\quad \Lie_\T\Deebar=\Deebar\Lie_\T.
\end{equation*}
The first of these identities implies that the image of $C^\infty(\N;\Wedge^q\Kbar^*\otimes F_\N)$ by $\Lie_\T$ is contained in $C^\infty(\N;\Wedge^q\Kbar^*\otimes F_\N)$.
With these definitions, $\Deebar$ as an operator
\begin{equation*}
\Deebar :
\begin{matrix}
C^\infty(\N;\Wedge^q\Kbar^*\otimes F_\N)\\ \oplus \\ C^\infty(\N;\Wedge^{q-1}\Kbar^*\otimes F_\N)
\end{matrix}
\to
\begin{matrix}
C^\infty(\N;\Wedge^{q+1}\Kbar^*\otimes F_\N)\\ \oplus \\ C^\infty(\N;\Wedge^q\Kbar^*\otimes F_\N)
\end{matrix}.
\end{equation*}
is given by the matrix in \eqref{DeeAsMatrix} with the new meanings for $\deebarb$ and $\Lie_\T$.

Assume that there is a $\T$-invariant Riemannian metric on $\N$,  that $\rr$ has be chosen so that $a_\rr$ is constant, that $\Kbar_\rr$ is orthogonal to $\T$, and that $\T$ has unit length. Then the term involving $\deebarb a_\rr$ in the matrix \eqref{DeeAsMatrix} is absent, and since $\Dee^2=0$,
\begin{equation*}
\Lie_\T\deebar_b=\deebarb\Lie_\T.
\end{equation*}
Write $h_{\smash[t]{\Veebar}^*}$ for the metric induced on the bundles $\Wedge^q\smash[t]{\Veebar}^*$ or $\Wedge^q\Kbar^*$.

If $\eta_\mu$, $\mu=1,\dotsc,k$ is a local frame of $F_\N$ over an open set $U\subset \N$ and $\phi$ is a local section of $\Wedge^q\smash[t]{\Veebar}^*\otimes F_\N$ over $U$, then for some smooth sections $\phi^\mu$ of $\Wedge^q\smash[t]{\Veebar}^*$ and $\omega^\nu_\mu$ of $\smash[t]{\Veebar}^*$ over $U$,
\begin{equation*}
\phi= \sum_\mu \phi^\mu \otimes \eta_\mu,\quad \Deebar \sum_\mu \phi^\mu \otimes \eta_\mu = \sum_\nu (\Deebar\phi^\nu + \sum_\mu \omega^\nu_\mu \wedge \phi^\mu)\otimes \eta_\nu.
\end{equation*}
This gives
\begin{equation*}
\deebarb \sum_\mu \phi^\mu \otimes \eta_\mu = \sum_\nu (\deebarb\phi^\nu + \sum_\mu \Pi_\rr\omega^\nu_\mu \wedge \phi^\mu)\otimes \eta_\nu
\end{equation*}
and
\begin{equation*}
\Lie_\T \sum_\mu \phi^\mu \otimes \eta_\mu = \sum_\nu (\Lie_\T\phi^\nu + \sum_\mu \langle\omega^\nu_\mu,\T\rangle \phi^\mu)\otimes \eta_\nu.
\end{equation*}

Suppose now that $h_F$ is a Hermitian metric on $F$. With this metric and the metric $h_{\smash[t]{\Veebar}^*}$ we get Hermitian metrics $h$ on each of the bundles $\Wedge^q\smash[t]{\Veebar}^*\otimes F_\N$. If $\eta_\mu$ is an orthonormal frame of $F_\N$ and $\phi=\sum\phi^\mu\otimes\eta_\mu$, $\psi=\sum\psi^\mu\otimes\eta_\mu$ are sections of $\Wedge^q\smash[t]{\Veebar}^*\otimes F_\N$, then
\begin{equation*}
h(\phi,\psi) = \sum_\nu h_{\smash[t]{\Veebar}^*}(\phi^\mu,\psi^\mu).
\end{equation*}
Therefore
\begin{align*}
h(\Lie_\T\phi,\psi)&+h(\phi,\Lie_\T\psi)\\
&= \sum_\nu h_{\smash[t]{\Veebar}^*}(\Lie_\T\phi^\nu + \sum_\mu \langle\omega^\nu_\mu,\T\rangle \phi^\mu,\psi^\nu)
+ \sum_\mu h_{\smash[t]{\Veebar}^*}(\phi^\mu, \Lie_\T\psi^\mu + \langle\omega^\mu_\nu,\T\rangle \psi^\nu)\\
&= \sum_\nu \T h_{\smash[t]{\Veebar}^*}(\phi^\nu,\psi^\nu) + \sum_{\mu,\nu}  ( \langle\omega^\nu_\mu,\T\rangle + \overline{\langle\omega^\mu_\nu,\T\rangle}) h_{\smash[t]{\Veebar}^*}(\phi^\mu,\psi^\nu)\\
&= \T h(\phi,\psi) + \sum_{\mu,\nu}  ( \langle\omega^\nu_\mu,\T\rangle + \overline{\langle\omega^\mu_\nu,\T\rangle}) h_{\smash[t]{\Veebar}^*}(\phi^\mu,\psi^\nu).
\end{align*}
Thus $\T h(\phi,\psi) = h(\Lie_\T\phi,\psi) +h(\phi,\Lie_\T\psi)$ if and only if
\begin{equation}\label{TangencyOfTE}
\langle\omega^\nu_\mu,\T\rangle + \overline{\langle\omega^\mu_\nu,\T\rangle}=0\text{ for all }\mu,\nu.
\end{equation}
This condition is \eqref{ExactMetricCondition}; just note that by the definition of $\Deebar$, the forms $(\Phi^*)^{-1}\omega^\nu_\mu$ in \eqref{ExactMetricCondition} are the forms that we are denoting $\omega^\nu_\mu$ here. Thus \eqref{TangencyOfTE} holds if and only if $h_F$ is an exact Hermitian metric, see Definition \eqref{ExactMetric}.

Consequently,

\begin{lemma} The statement
\begin{equation}\label{LieTbis}
\T h(\phi,\psi) = h(\Lie_\T\phi,\psi)+h(\phi,\Lie_\T\psi)\quad \forall \phi,\psi\in C^\infty(\N;\Wedge^q\smash[t]{\Veebar}^*\otimes F_\N)
\end{equation}
holds if and only the Hermitian metric $h_F$ is exact.
\end{lemma}

\section{Spectrum}\label{sSpectrum}

Suppose that $\Veebar$ admits an invariant Hermitian metric. Let $\rr$ be a defining function for $\N$ in $\M$ such that $a_\rr$ is constant. By  Lemma \eqref{Invariances} $\Kbar_\rr$ is $\T$-invariant, so the restriction of the metric to this subbundle gives a $\T$-invariant metric; we use the induced metric on the bundles $\Wedge^q\Kbar^*$ in the following. As in the proof of Lemma~\ref{WithInvariantMetric}, there is a $\T$-invariant density $\m_0$ on $\N$. 

Let $\rho:F\to\M$ be a Hermitian holomorphic vector bundle, assume that the Hermitian metric of $F$ is exact, so with the induced metric $h$ on the vector bundles $\Wedge^q\smash[t]{\Veebar}^*\otimes F_\N$, \eqref{LieTbis} holds. We will write $F$ in place of $F_\N$. 

Let $\deebarb^\star$ be the formal adjoint of the $\bdeebar$ operator \eqref{deebarbE} with respect to the inner on the bundles $\Wedge^q\Kbar^*\otimes F$ and the density $\m_0$, and let $\Laplacian_{b,q} = \deebarb\deebarb^\star + \deebarb^\star\deebarb$ be the formal $\deebarb$-Laplacian. Since $-\im\Lie_\T$ is formally selfadjoint and commutes with $\deebarb$, $\Lie_\T$ commutes with $\Laplacian_{b,q}$. Let
\begin{equation*}
\Ha^q_{\deebarb}(\N;F)=\ker\Laplacian_{b,q}=\set{\phi\in L^2(\N;\Wedge^q\Kbar^*\otimes F):\Laplacian_{b,q}\phi=0}
\end{equation*}
and let
\begin{equation*}
\Dom_q(\Lie_\T)=\set{\phi\in \Ha^q_{\deebarb}(\N;F)\text{ and }\Lie_\T\phi\in \Ha^q_{\deebarb}(\N;F)}.
\end{equation*}
The spaces $\Ha^q_{\deebarb}(\N;F)$ may be of infinite dimension, but in any case they are closed subspaces of $L^2(\N;\Wedge^q\Kbar^*\otimes F)$, so they may be regarded as Hilbert spaces on their own right. If $\phi\in \Ha^q_{\deebarb}(\N;F)$, the condition $\Lie_\T\phi\in \Ha^q_{\deebarb}(\N;F)$ is equivalent to the condition
\begin{equation*}
\Lie_\T\phi\in L^2(\N;\Wedge^q\Kbar^*\otimes F).
\end{equation*}
So we have a closed operator
\begin{equation}\label{LieOnB-Harmonic}
-\im \Lie_\T:\Dom_q(\Lie_\T)\subset \Ha^q_{\deebarb}(\N;F)\to \Ha^q_{\deebarb}(\N;F).
\end{equation}
The fact that $\Laplacian_{b,q}-\Lie_\T^2$ is elliptic, symmetric, and commutes with $\Lie_\T$ implies that \eqref{LieOnB-Harmonic} is a selfadjoint Fredholm operator with discrete spectrum (see \cite[Theorem 2.5]{Me9}).

\begin{definition}\label{BCohomologyWithCoeffs}
Let $\spec^q_0(-\im \Lie_\T)$ be the spectrum of the operator \eqref{LieOnB-Harmonic}, and let $\Ha^q_{\deebarb,\tau}(\N;F)$ be the eigenspace of $-\im\Lie_\T$ in $\Ha^q_{\deebarb}(\N;F)$ corresponding to the eigenvalue $\tau$.
\end{definition}

Let $\pmb\tau$ denote the principal symbol of $-\im\T$. Then the principal symbol of $\Lie_\T$ acting on sections of $\Wedge^q\Kbar^*$ is $\pmb\tau\Id$. Because $\Laplacian_{b,q}-\Lie_\T^2$ is elliptic, $\Char(\Laplacian_{b,q})$, the characteristic variety of $\Laplacian_{b,q}$, lies in $\pmb\tau\ne 0$. Let
\begin{equation*}
\Char^\pm(\Laplacian_{b,q})=\set{\nu\in\Char (\Laplacian_{b,q}):\pmb \tau(\nu)\gtrless 0}.
\end{equation*}
By \cite[Theorem 4.1]{Me9}, if $\Laplacian_{b,q}$ is microlocally hypoelliptic on $\Char^\pm(\Laplacian_{b,q})$, then 
\begin{equation*}
\set{\tau\in \spec^q_0(-\im \Lie_\T):\tau\gtrless0}
\end{equation*}
is finite. We should perhaps point out that $\Char(\Laplacian_{b,q})$ is equal to the characteristic variety, $\Char(\Kbar_r)$, of the CR structure. 

As a special case consider the situation where $F$ is the trivial line bundle. Let $\theta_\rr$ be the real $1$-form on $\N$ which vanishes on $\Kbar_\rr$ and satisfies $\langle\theta_\rr,\T\rangle =1$; thus $\theta_\rr$ is smooth, spans $\Char(\Kbar_\rr)$, and has values in $\Char^+(\Kbar_\rr)$. The Levi form of the structure is
\begin{equation*}
\Levi_{\theta_\rr}(v,w)=-\im d\theta_\rr(v,\overline w), \quad v,\ w\in \K_{\rr,p},\ p\in \N.
\end{equation*}
Suppose that $\Levi_{\theta_\rr}$ is nondegenerate, with $k$ positive and $n-k$ negative eigenvalues. It is well known that then $\Laplacian_{b,q}$ is microlocally hypoelliptic at $\nu\in\Char\K_\rr$ for all $q$ except if $q=k$ and $\pmb \tau(\nu)<0$ or if $q=n-k$ and $\pmb \tau(\nu)>0$.

Then the already mentioned Theorem~4.1 of \cite{Me9} gives:

\begin{theorem}[{\cite[Theorem 6.1]{Me9}}]\label{WeakVanishing}
Suppose that $\Veebar$ admits a Hermitian metric and that for some defining function $\rr$ such that $\a_\rr$ is constant, $\Levi_{\theta_\rr}$ is nondegenerate with $k$ positive and $n-k$ negative eigenvalues. Then
\begin{enumerate}
\item $\spec_0^q(-\im \Lie_\T)$ is finite if $q\ne k,\ n-k$;
\item $\spec_0^k(-\im\Lie_\T)$ contains only finitely many positive elements, and \item $\spec_0^{n-k}(-\im\Lie_\T)$ contains only finitely many negative elements.
\end{enumerate}
\end{theorem}

\section{Indicial cohomology}\label{sIndicialCohomology}

Suppose that there is a $\T$-invariant Hermitian metric $\tilde h$ on $\Veebar$. By Lemma~\ref{WithInvariantMetric} there is a defining function $\rr$ such that $\langle\beta_\rr,\T\rangle$ is constant, equal to $a_\av-\im$. Therefore $\Kbar_\rr$ is $\T$-invariant. Let $h$ be the metric on $\Veebar$ which coincides with $\tilde h$ on $\Kbar_\rr$, makes the decomposition $\Veebar=\Kbar_\rr\oplus \Span_\C\T$ orthogonal, and for which $\T$ has unit length. The metric $h$ is $\T$-invariant. We fix $\rr$ and such a metric, and let $\m_0$ be the Riemannian measure associated with $h$. The decomposition \eqref{DecompositionOfVeebar} of $\Wedge^q\smash[t]{\Veebar}^*$ is an orthogonal decomposition.

Recall that $\Dbar(\sigma)\phi=\Deebar \phi+\im\sigma\beta_\rr\wedge \phi$. Since $a_\rr=a_\av$ is constant (in particular CR),
\begin{equation*}
\Dbar(\sigma)(\phi^0 + \im \tilde \beta_\rr\wedge \phi^1) =
\deebarb\phi_0+\im\tilde \beta_\rr\wedge \big[\big(\Lie_\T+(1+\im a_\av)\sigma\big)\phi^0 -\deebarb\phi^1\big]
\end{equation*}
if $\phi^0\in C^\infty(\N;\Wedge^q\Kbar_\rr^*)$ and $\phi^1\in C^\infty(\N;\Wedge^{q-1}\Kbar_\rr^*)$. So $\Dbar(\sigma)$ can be regarded as the operator
\begin{equation}\label{CalDasMarix}
\Dbar(\sigma)=\begin{bmatrix}
\deebarb & 0\\
\Lie_\T+(1+\im a_\av)\sigma & -\deebarb
\end{bmatrix}:
\begin{matrix}
C^\infty(\N;\Wedge^q\Kbar_\rr^*) \\ \oplus\\ C^\infty(\N;\Wedge^{q-1}\Kbar_\rr^*)
\end{matrix}
\to
\begin{matrix}
C^\infty(\N;\Wedge^{q+1}\Kbar_\rr^*) \\ \oplus\\ C^\infty(\N;\Wedge^{q}\Kbar_\rr^*).
\end{matrix}
\end{equation}
Since the subbundles $\Wedge^q\Kbar_\rr$ and $\tilde \beta \wedge \Wedge^{q-1} \Kbar_\rr$ are orthogonal with respect to the metric induced by $h$ on $\Wedge^q\Veebar$, the formal adjoint of $\Dbar(\sigma)$ with respect to this metric and the density $\m_0$ is
\begin{equation*}
\Dbar(\sigma)^\star=
\begin{bmatrix}
\deebarb^\star & -\Lie_\T+(1-\im a_\av)\overline \sigma \\
0& -\deebarb^\star
\end{bmatrix}:
\begin{matrix}
C^\infty(\N;\Wedge^{q+1}\Kbar_\rr^*) \\ \oplus\\ C^\infty(\N;\Wedge^{q}\Kbar_\rr^*)
\end{matrix}
\to
\begin{matrix}
C^\infty(\N;\Wedge^q\Kbar_\rr^*) \\ \oplus\\ C^\infty(\N;\Wedge^{q-1}\Kbar_\rr^*)
\end{matrix}
\end{equation*}
where $\deebarb^\star$ is the formal adjoint of $\deebarb$. So the Laplacian, $\Laplacian_{\Dbar(\sigma),q}$, of the $\Dbar(\sigma)$-complex is the diagonal operator with diagonal entries $P_q(\sigma)$, $P_{q-1}(\sigma)$ where
\begin{equation*}
P_q(\sigma)=\Laplacian_{b,q}+(\Lie_\T+(1+\im a_\av)\sigma)(-\Lie_\T+(1-\im a_\av)\overline \sigma)
\end{equation*}
acting on $C^\infty(\N;\Wedge^q\Kbar_\rr^*)$ and $P_{q-1}(\sigma)$ is the ``same'' operator, acting on sections of $\Wedge^{q-1}\Kbar_\rr^*$; recall that $\Lie_\T$ commutes with $\deebarb$ and since $\Lie_\T^\star=-\Lie_\T$, also with $\deebarb^\star$, and that $a_\av$ is constant. Note that $P_q(\sigma)$ is an elliptic operator.

Suppose that $\phi\in C^\infty(\N;\Wedge^q\Kbar_\rr^*)$ is a nonzero element of $\ker P_q(\sigma)$; the complex number $\sigma$ is fixed. Since $P_q(\sigma)$ is elliptic, $\ker P_q(\sigma)$ is a finite dimensional space, invariant under $-\im \Lie_\T$ since the latter operator commutes with $P_q(\sigma)$. As an operator on $\ker P_q(\sigma)$, $-\im \Lie_\T$ is selfadjoint, so there is a decomposition of $\ker P_q(\sigma)$ into eigenspaces of $-\im\Lie_\T$. Thus
\begin{equation*}
\phi=\sum_{j=1}^N \phi_j, \quad -\im \Lie_\T\phi_j=\tau_j\phi_j
\end{equation*}
where the $\tau_j$ are distinct real numbers and $\phi_j\in \ker P_q(\sigma)$, $\phi_j\ne 0$. In particular,
\begin{equation*}
\Laplacian_{b,q}\phi_j + (\Lie_\T+(1+\im a_\av)\sigma)(-\Lie_\T+(1-\im a_\av)\overline \sigma)\phi_j = 0,
\end{equation*}
for each $j$, that is,
\begin{equation*}
\Laplacian_{b,q}\phi_j+ |\im\tau_j+(1+\im a_\av)\sigma|^2\phi_j = 0.
\end{equation*}
Since $\Laplacian_{b,q}$ is a nonnegative operator and $\phi_j\ne 0$, $\im\tau_j+(1+\im a_\av)\sigma=0$ and $\phi_j\in \ker\Laplacian_{b,q}$. Since $\sigma$ is fixed, all $\tau_j$ are equal, which means that $N=1$. Conversely, if $\phi\in C^\infty(\N;\Wedge^q\Kbar_\rr^*)$ belongs to $\ker \Laplacian_{b,q}$ and $-\im\Lie_\T\phi=\tau \phi$, then $P_q(\sigma)\phi=0$ with $\sigma$ such that $\tau=(\im-a_\av)\sigma$.

Let $\Ha^q_{\Dbar(\sigma)}(\N)$ be the kernel of $\Laplacian_{\Dbar(\sigma),q}$.

\begin{theorem}\label{TheCohomology}
Suppose that $\Veebar$ admits a $\T$-invariant metric and let $\rr$ be a defining function for $\N$ in $\M$ such that $\langle \beta_\rr,\T\rangle=a_\av-\im$ is constant. Then
\begin{equation*}
\spec_{b,\N}^q(\bdeebar) = (\im-a_\av)^{-1}\spec_0^q(-\im \Lie_\T)\cup (\im-a_\av)^{-1}\spec_0^{q-1}(-\im \Lie_\T),
\end{equation*}
and if $\sigma\in \spec_{b,\N}^q(\bdeebar)$, then, with the notation in Definition \ref{BCohomologyWithCoeffs}
\begin{equation*}
\Ha^q_{\Dbar(\sigma)}(\N)=\Ha^q_{\deebarb,\tau(\sigma)}(\N)\oplus \Ha^{q-1}_{\deebarb,\tau(\sigma)}(\N)
\end{equation*}
with $\tau(\sigma)=(\im-a_\av)\sigma$.
\end{theorem}

If the CR structure $\overline \K_\rr$ is nondegenerate, Proposition~\ref{WeakVanishing} gives more specific information on $\spec_{b,\N}^q(\bdeebar)$. In particular,
\begin{proposition}
With the hypotheses of Theorem~\ref{TheCohomology}, suppose that $\Levi_{\theta_\rr}$ is nondegenerate with $k$ positive and $n-k$ negative eigenvalues. If $k>0$, then $\spec_{b,\N}^0 \subset \set{\sigma\in \C:\Im\sigma\leq 0}$, and if $n-k>0$, then $\spec_{b,\N}^0(\bdeebar) \subset \set{\sigma\in \C:\Im\sigma\geq 0}$.
\end{proposition}

\begin{remark}
The $b$-spectrum of the Laplacian of the $\bdeebar$-complex in any degree can be described explicitly in terms of the joint spectra $\spec(-\im\Lie_\T,\Laplacian_{b,q})$. We briefly indicate how. With the metric $h$ and defining function $\rr$ as in the first paragraph of this section, suppose that $h$ is extended to a metric on $\bT^{0,1}\M$. This gives a Riemannian $b$-metric on $\M$ that in turn gives a $b$-density $\m$ on $\M$. With these we get formal adjoints $\bdeebar^\star$ whose indicial families $\Dbar^\star(\sigma)$ are related to those of $\bdeebar$ by
\begin{equation*}
\Dbar^\star(\sigma) = \widehat{\bdeebar^\star}(\sigma)= [\widehat{\bdeebar}(\overline \sigma)]^\star = \Dbar(\overline \sigma)^\star.
\end{equation*}
By \eqref{CalDasMarix},
\begin{equation*}
\Dbar^\star(\sigma)=
\begin{bmatrix}
\deebarb^\star & -\Lie_\T+(1-\im a_\av) \sigma \\
0& -\deebarb^\star
\end{bmatrix}.
\end{equation*}
Using this one obtains that the indicial family of the Laplacian $\Laplacian_q$ of the $\bdeebar$-complex in degree $q$ is a diagonal operator with diagonal entries $P'_q(\sigma)$, $P'_{q-1}(\sigma)$ with
\begin{equation*}
P'_q(\sigma)=\Laplacian_{b,q}+(\Lie_\T+(1+\im a_\av)\sigma)(-\Lie_\T+(1-\im a_\av)\sigma)
\end{equation*}
and the analogous operator in degree $q-1$. The set $\spec_b(\Laplacian_q)$ is the set of values of $\sigma$ for which either $P'_q(\sigma)$ or $P'_{q-1}(\sigma)$ is not injective. These points can written in terms of the points $\spec(-\im\Lie_\T,\Laplacian_{b})$ as asserted. In particular one gets
\begin{equation*}
\spec_b(\Laplacian_q)\subset \set{\sigma: |\Re \sigma|\leq |a_\av||\Im\sigma|}
\end{equation*}
with $\spec_{b,\N}^q(\bdeebar)$ being a subset of the boundary of the set on the right.
\end{remark}

We now discuss the indicial cohomology sheaf of $\bdeebar$, see Definition \ref{CohomologySheafs}. We will show:

\begin{proposition}\label{bdeebarCohomologySheaf}
Let $\sigma_0\in \spec_{b,\N}^q(\bdeebar)$. Every element of the stalk of $\mathfrak H^q_{\bdeebar}(\N)$ at $\sigma_0$ has a representative of the form
\begin{equation*}
\frac{1}{\sigma-\sigma_0}
\begin{bmatrix}
\phi^0 \\ 0
\end{bmatrix}
\end{equation*}
where $\phi^0\in \Ha^q_{\deebarb,\tau_0}(\N)$, $\tau_0= (\im-a_\av)\sigma_0$.
\end{proposition}

\begin{proof}
Let
\begin{equation}\label{RepresentativeA}
\phi(\sigma)=\sum_{k=1}^\mu \frac{1}{(\sigma-\sigma_0)^k}
\begin{bmatrix}
\phi^0_k \\ \phi^1_k
\end{bmatrix}
\end{equation}
represent an element in the stalk at $\sigma_0$ of the sheaf of germs of $C^\infty(\N;\Wedge^q\smash[t]{\Veebar}^*\otimes F)$-valued meromorphic functions on $\C$ modulo the subsheaf of holomorphic elements. Letting $\alpha=1+\im a_\av$ we have
\begin{equation*}
\Dbar(\sigma)\phi(\sigma) =
\sum_{k=1}^\mu \frac{1}{(\sigma-\sigma_0)^k}
\begin{bmatrix}
\deebarb \phi^0_k \\
\big(\Lie_\T + \alpha\sigma_0\big) \phi^0_k-\deebarb \phi^1_k
\end{bmatrix}
+ \sum_{k=0}^{\mu-1} \frac{\alpha}{(\sigma-\sigma_0)^k}
\begin{bmatrix}
0 \\
\phi^0_{k+1}
\end{bmatrix},
\end{equation*}
so the condition that $\Dbar(\sigma)\phi(\sigma)$ is holomorphic is equivalent to
\begin{equation}\label{TopEqs}
\deebarb\phi^0_k=0,\ k=1,\dotsc,\mu
\end{equation}
and
\begin{equation}\label{BottomEqs}
\begin{gathered}
(\Lie_\T + \alpha\sigma_0) \phi^0_\mu-\deebarb \phi^1_\mu = 0,\\ (\Lie_\T + \alpha\sigma_0) \phi^0_k-\deebarb \phi^1_k + \alpha\phi^0_{k+1}=0,\ k=1,\dotsc,\mu-1.
\end{gathered}
\end{equation}

Let $P_{q'}=\Laplacian_{b,q'}-\Lie_\T^2$ in any degree $q'$. For any $(\tau,\lambda)\in \R^2$ and $q'$ let
\begin{equation*}
\mathcal E^{q'}_{\tau,\lambda}=\set{\psi\in C^\infty(\N;\Wedge^{q'}\smash[t]{\Veebar}^*\otimes F):P_{q'}\psi=\lambda\psi,\ -\im \Lie_\T \psi=\tau\psi}.
\end{equation*}
This space is zero if $(\tau,\lambda)$ is not in the joint spectrum $\Sigma^{q'} = \spec^{q'}(-\im\Lie_\T,P_{q'})$. Each $\phi^i_k$ decomposes as a sum of elements in the spaces $\mathcal E^{q-i}_{\tau,\lambda}$, $(\tau,\lambda)\in \Sigma^{q-i}$. Suppose that already $\phi^i_k\in \mathcal E^{q-i}_{\tau,\lambda}$:
\begin{equation*}
P_{q-i}\phi^i_k=\lambda \phi^i_k,\quad -\im\Lie_\T \phi^i_k=\tau\phi^i_k,\quad i=0,1,\ k=1,\dotsc,\mu.
\end{equation*}
Then \eqref{BottomEqs} becomes
\begin{equation}\label{BottomEqsEigen}
\begin{gathered}
(\im\tau + \alpha\sigma_0) \phi^0_\mu-\deebarb \phi^1_\mu = 0,\\ (\im\tau + \alpha\sigma_0) \phi^0_k-\deebarb \phi^1_k + \alpha\phi^0_{k+1}=0,\ k=1,\dotsc,\mu-1.
\end{gathered}
\end{equation}
If $\tau\ne \tau_0$, then $\im\tau + \alpha\sigma_0\ne 0$, and we get $\phi^0_k=\deebarb \psi^0_k$ for all $k$ with
\begin{equation*}
\psi^0_k = \sum_{j=0}^{\mu-k} \frac{(-\alpha)^j}{(\im\tau + \alpha\sigma_0)^{j+1}}\phi^1_{k+j}.
\end{equation*}
Trivially
\begin{equation*}
\big(\Lie_\T + \alpha\sigma_0\big) \psi^0_\mu = \phi^1_\mu
\end{equation*}
and also
\begin{equation*}
\big(\Lie_\T + \alpha\sigma_0\big) \psi^0_k + \alpha\psi^0_{k+1} = \phi^1_k,\quad k=1,\dotsc,\mu-1,
\end{equation*}
so
\begin{equation*}
\phi(\sigma) - \Dbar(\sigma)\sum_{k=1}^\mu \frac{1}{(\sigma-\sigma_0)^k}
\begin{bmatrix} \psi^0_k\\
0
\end{bmatrix}=0
\end{equation*}
modulo an entire element.

Suppose now that the $\phi^i_k$ are arbitrary and satisfy \eqref{TopEqs}-\eqref{BottomEqs}. The sum
\begin{equation}\label{FourierSeries}
\phi^i_k=\sum_{(\tau,\lambda)\in \Sigma^{q-i}} \phi^i_{k,\tau,\lambda},\quad \phi^i_{k,\tau,\lambda}\in \mathcal E^{q-i}_{\tau,\lambda}
\end{equation}
converges in $C^\infty$, indeed for each $N$ there is $C_{i,k,N}$ such that
\begin{equation}\label{RapidDecay}
\sup_{p\in \N}\|\phi^i_{k,\tau,\lambda}(p)\|\leq C_{i,k,N}(1+\lambda)^{-N}\quad\text{ for all }\tau,\lambda.
\end{equation}
Since $\Dbar(\sigma)$ preserves the spaces $\mathcal E^q_{\tau,\lambda}\oplus \mathcal E^{q-1}_{\tau,\lambda}$, the relations \eqref{BottomEqsEigen} hold for the $\phi^i_{k,\tau,\lambda}$ for each $(\tau,\lambda)$. Therefore,
with
\begin{equation}\label{PsiFourierSeries}
\psi^0_k = \sum_{\substack{(\tau,\lambda)\in\Sigma^{q-1}\\\tau\ne \tau_0}}\sum_{j=0}^{\mu-k} \frac{(-\alpha)^j}{(\im\tau + \alpha\sigma_0)^{j+1}}\phi^1_{k+j,\tau,\lambda}
\end{equation}
we have formally that
\begin{equation*}
\phi(\sigma)-\Dbar(\sigma) \sum_{k=1}^\mu \frac{1}{(\sigma-\sigma_0)^k}
\begin{bmatrix} \psi_k\\
0
\end{bmatrix} = \sum_{k=1}^\mu \frac{1}{(\sigma-\sigma_0)^k}
\begin{bmatrix} \tilde \phi^0_k\\
\tilde \phi^1_k
\end{bmatrix}
\end{equation*}
with
\begin{equation}\label{AtTau0}
\tilde \phi^i_k=\sum_{\substack{(\tau,\lambda)\in \Sigma^{q-1}\\\tau=\tau_0}} \phi^i_{k,\tau,\lambda},\quad \phi^i_{k,\tau,\lambda}\in \mathcal E^{q-i}_{\tau,\lambda}.
\end{equation}
However, the convergence in $C^\infty$ of the series \eqref{PsiFourierSeries} is questionable since there may be a sequence $\set{(\tau_\ell,\lambda_\ell)}_{\ell=1}^\infty \subset \spec(-\im\Lie_\T,P_{q-1})$ of distinct points such that $\tau_\ell\to\tau_0$ as $\ell\to\infty$, so that the denominators $\im \tau_\ell+\alpha\sigma_0$ in the formula for $\psi^0_k$ tend to zero so fast that for some nonnegative $N$, $\lambda_\ell^{-N}/(\im \tau_\ell+\alpha\sigma_0)$ is unbounded. To resolve this difficulty we will first show that $\phi(\sigma)$ is $\Dbar(\sigma)$-cohomologous (modulo holomorphic terms) to an element of the same form as $\phi(\sigma)$ for which in the series \eqref{FourierSeries} the terms $\phi^i_{k,\tau,\lambda}$ vanish if $\lambda-\tau^2>\e$; the number $\e>0$ is chosen so that
\begin{equation}\label{ChoiceOfEps}
(\tau_0,\lambda)\in \Sigma^q\cup\Sigma^{q-1} \implies \lambda=\tau_0^2\text{ or }\lambda\geq\tau_0^2+\e.
\end{equation}
Recall that $\spec^{q'}(-\im\Lie_\T,P_{q'})\subset \set{(\tau,\lambda):\lambda\geq \tau^2}$.

For any $V\subset \bigcup_{q'}\Sigma^{q'}$ let
\begin{equation*}
\Pi^{q'}_V:L^2(\N;\Wedge^{q'}\smash[t]{\Veebar}^*\otimes F)\to L^2(\N;\Wedge^{q'}\smash[t]{\Veebar}^*\otimes F)
\end{equation*}
be the orthogonal projection on $\bigoplus_{(\tau,\lambda)\in V} \mathcal E^{q'}_{\tau,\lambda}$. If $\psi\in C^\infty(\N;\Wedge^{q'}\smash[t]{\Veebar}^*\otimes F)$, then the series
\begin{equation*}
\Pi^{q'}_V \psi=\sum_{(\tau,\lambda)\in V}\psi_{\tau,\lambda}, \quad\psi_{\tau,\lambda}\in \mathcal E^{q'}_{\tau,\lambda}
\end{equation*}
converges in $C^\infty$. It follows that $\Laplacian_{b,q'}$ and $\Lie_\T$ commute with $\Pi^{q'}_V$ and that $\deebarb\Pi^{q'}_V=\Pi^{q'+1}_V\deebarb$. Since the $\Pi^{q'}_V$ are selfadjoint, also $\deebar_b^\star\Pi^{q'+1}_V = \Pi^{q'}_V\deebarb^\star$.

Let
\begin{equation*}
U=\set{(\tau,\lambda)\in \Sigma^q\cup \Sigma^{q-1}:\lambda <\tau^2+\e},\quad U^c=\Sigma^q\cup \Sigma^{q-1}\minus U.
\end{equation*}
Then, for any sequence $\set{(\tau_\ell,\lambda_\ell)}\subset U$ of distinct points we have $|\tau_\ell|\to\infty$ as $\ell\to\infty$. Define
\begin{equation*}
G^{q'}_{U^c}\psi =\sum_{(\tau,\lambda)\in U^c}\frac{1}{\lambda-\tau^2}\psi_{\tau,\lambda}
\end{equation*}
In this definition the denominators $\lambda-\tau^2$ are bounded from below by $\e$, so $G^{q'}_{U^c}$ is a bounded operator in $L^2$ and maps smooth sections to smooth sections because the components of such sections satisfy estimates as in \eqref{RapidDecay}. The operators are analogous to Green operators: we have
\begin{equation}\label{PseudoGreen1}
\Laplacian_{b,q'}G^{q'}_{U^c} = G^{q'}_{U^c} \Laplacian_{b,q'}=\Id -\Pi^{q'}_U
\end{equation}
so if $\deebarb\psi=0$, then
\begin{equation}\label{PseudoGreen2}
\Laplacian_{b,q'}G^{q'}_{U^c}\psi = \deebarb\deebarb^\star G^{q'}_{U^c}\psi
\end{equation}
since $\deebarb G^{q'}_{U^c}= G^{q'+1}_{U^c}\deebarb$.

Write $\phi(\sigma)$ in \eqref{RepresentativeA} as
\begin{equation*}
\phi(\sigma) =\Pi_{U^c}\phi(\sigma)+ \Pi_{U}\phi(\sigma)
\end{equation*}
where
\begin{equation*}
\Pi_{U^c}\phi(\sigma) =\sum_{k=1}^\mu \frac{1}{(\sigma-\sigma_0)^k}
\begin{bmatrix}
\Pi^q_{U^c}\phi^0_k \\ \Pi^{q-1}_{U^c}\phi^1_k
\end{bmatrix},\quad
\Pi_U\phi(\sigma)=
\sum_{k=1}^\mu \frac{1}{(\sigma-\sigma_0)^k}
\begin{bmatrix}
\Pi^q_U\phi^0_k \\ \Pi^{q-1}_U\phi^1_k
\end{bmatrix}.
\end{equation*}
Since $\Dbar(\sigma)\phi(\sigma)$ is holomorphic, so are $\Dbar(\sigma)\Pi_{U^c}\phi(\sigma)$ and $\Dbar(\sigma)\Pi_U\phi(\sigma)$.

We show that $\Pi_{U^c}\phi(\sigma)$ is exact modulo holomorphic functions. Using \eqref{TopEqs}, \eqref{PseudoGreen1}, and \eqref{PseudoGreen2}, $\Pi^q_{U^c}\phi^0_k = \deebarb^\star \deebarb \Pi^q_{U^c}\phi^0_k$. Then
\begin{equation*}
\Pi_{U^c}\phi(\sigma)-\Dbar(\sigma)\sum_{k=1}^\mu \frac{1}{(\sigma-\sigma_0)^k}
\begin{bmatrix}
\deebarb^\star G^q_{U^c}\Pi^q_U\phi^0_k \\ 0
\end{bmatrix}
=\sum_{k=1}^\mu \frac{1}{(\sigma-\sigma_0)^k}
\begin{bmatrix}
0 \\ \hat\phi^1_k
\end{bmatrix}
\end{equation*}
modulo a holomorphic term for some $\hat\phi^1_k$ with $\Pi^{q-1}_{U^c}\hat\phi^1_k=\hat\phi^1_k$. The element on the right is $\Dbar(\sigma)$-closed modulo a holomorphic function, so its components satisfy \eqref{TopEqs}, \eqref{BottomEqs}, which give that the $\tilde\phi^1_k$ are $\deebarb$-closed. Using again \eqref{PseudoGreen1} and \eqref{PseudoGreen2} we see that $\Pi_{U^c}\phi(\sigma)$ represent an exact element.

We may thus assume that $\Pi^q_{U^c}\phi(\sigma)=0$. If this is the case, then the series \eqref{PsiFourierSeries} converges in $C^\infty$, so $\phi(\sigma)$ is cohomologous to the element
\begin{equation*}
\tilde\phi(\sigma)=\sum_{k=1}^\mu \frac{1}{(\sigma-\sigma_0)^k}
\begin{bmatrix}
\tilde \phi^0_k \\ \tilde \phi^1_k
\end{bmatrix}
\end{equation*}
where the $\tilde \phi^i_k$ are given by \eqref{AtTau0} and satisfy $\Pi^{q-i}_{U^c}\tilde \phi^i_k=0$. By \eqref{ChoiceOfEps}, $\tilde \phi^i_k\in \mathcal E^{q-i}_{\tau_0,\tau_0^2}$. In particular, $\Laplacian_{b,q-i}\phi^i_k=0$.

Assuming now that already $\phi^i_k\in \mathcal E^{q-i}_{\tau_0,\tau_0^2}$, the formulas \eqref{BottomEqsEigen} give (since $\tau=\tau_0$ and $\im\tau_0+\alpha\sigma_0=0$)
\begin{equation*}
\deebarb\phi^1_\mu=0,\quad \phi^0_k = \deebarb \frac{1}{\alpha}\phi^1_{k-1},\ k=2,\dotsc,\mu.
\end{equation*}
Then
\begin{equation*}
\phi(\sigma)-\frac {1}{\alpha}\Dbar(\sigma)\sum_{k=2}^{\mu+1} \frac{1}{(\sigma-\sigma_0)^k}
\begin{bmatrix}
\phi^1_{k-1} \\ 0
\end{bmatrix}
=\frac{1}{\sigma-\sigma_0}
\begin{bmatrix}
\phi^0_1 \\ 0
\end{bmatrix}
\end{equation*}
with $\Laplacian_{b,q}\phi^0_1=0$.
\end{proof}

\appendix

\section{Totally characteristic differential operators}\label{TheAppendix}

We review here some basic definitions and notation concerning totally characteristic differential operators.

Let $E$, $F\to\M$ be vector bundles and let $\Diff^m(\M;E,F)$ be the space of differential operators $C^\infty(\M;E)\to C^\infty(\M;F)$ of order $m$. Then 
\begin{equation}\label{TotallyChar}
\display{300pt}{
$\Diff_b^m(\M;E,F)$, the space of totally characteristic differential operators of order $m$, consists of those elements $P\in \Diff^m(\M;E,F)$ with the property 
\begin{equation*}
\hspace*{-30pt}\rr^{-\nu}P\rr^\nu \in \Diff^m(\M;E,F),\quad \nu=1,\dotsc, m
\end{equation*}
\ie, $\rr^{-\nu}P\rr^\nu$ has coefficients smooth up to the boundary. }
\end{equation}

\medskip
Let $\pi:T^*\M\to\M$ and $\bpi:\bT^*\M\to\M$ be the canonical projections. Suppose $P\in \Diff^m_b(\M;E,F)$. Since $P$ is in particular a differential operator, it has a principal symbol
\begin{equation*}
\sym(P)\in C^\infty(T^*\M;\Hom(\pi^* E,\pi^*F)).
\end{equation*}
The fact that $P$ is totally characteristic implies that $\sym(P)$ lifts to a section
\begin{equation*}
\bsym(P)\in C^\infty(\bT^*\M;\Hom(\bpi^* E,\bpi^*F)),
\end{equation*}
the principal $b$-symbol of $P$, characterized by
\begin{equation}\label{TheBSymbol}
\bsym(P)(\ev^*\xi) = \sym(P)(\xi).
\end{equation}

If $P\in \Diff^m_b(\M;E,F)$, then $P$ induces a differential operator
\begin{equation}\label{Pb}
P_b\in \Diff^m_b(\M;E_{\partial\M},F_{\partial\M}),
\end{equation}
as follows. If $\phi\in C^\infty(\partial\M;E_{\partial\M})$, let $\tilde \phi \in C^\infty(\M;E)$ be an extension of $\phi$ and let
\begin{equation*}
P_b\phi = (P\tilde \phi)|_{\partial\M}.
\end{equation*}
The condition \eqref{TotallyChar} ensures that $P\tilde \phi|_{\partial\M}$ is independent of the extension of $\phi$ used. Clearly if $P$ and $Q$ are totally characteristic differential operators, then so is $PQ$, and
\begin{equation}\label{PQb}
(PQ)_b=P_bQ_b.
\end{equation}

The indicial family of $P\in \Diff^m_b(\M;E,F)$ is defined as follows. Fix a defining function $\rr$ for $\partial\M$. Then for any $\sigma\in \C$,
\begin{equation*}
P(\sigma)=\rr^{-\im \sigma}P\rr^{\im \sigma}\in \Diff^m_b(\M;E,F).
\end{equation*}
Let 
\begin{equation}\label{IndicialFamily}
\widehat P(\sigma)=P(\sigma)_b.
\end{equation}


\end{document}